\documentclass[12pt]{amsart}
\usepackage{style}

\newcommand{\RSAthreethree}{\cite[Proposition~3.3]{RSA}}
\newcommand{\RSAfourtwo}{\cite[Corollary~4.2]{RSA}}
\newcommand{\RSAsevenfour}{\cite[Corollary~7.4]{RSA}}
\newcommand{\RSAeighttwo}{\cite[Proposition~8.2]{RSA}}

\usepackage{tikz-cd}
\usetikzlibrary{decorations.pathreplacing}

\begin{document}

\title{Modules of the Temperley-Lieb algebra at zero}
\author{Eddy Li}
\address{MIT PRIMES, Massachusetts Institute of Technology, Massachusetts Avenue, Cambridge, MA 02139-4301, USA}
\email{eddylihere@gmail.com}

\author{Kenta Suzuki}
\address{Department of Mathematics, Princeton University, Fine Hall, Washington Road, Princeton, NJ 08544-1000, USA}
\email{kjsuzuki@princeton.edu}

\date{January 29, 2026}

\begin{abstract}
We explicitly describe the category of modules of the Temperley-Lieb algebra $\TL_n(\beta)$ under specialization $\beta=0$ for even $n$ in terms of a quiver algebra, analogous to a result of Berest-Etingof-Ginzburg. In particular, we explicitly construct an exact sequence of the standard modules of $\TL_n(0)$, which categorifies a numerical coincidence regarding the evaluation of the Jones polynomial at $t=-1$. We furthermore deduce a consequence in the representation theory of symmetric groups over characteristic two.
\end{abstract}

\maketitle

\section{Introduction}\label{intro2}
Temperley and Lieb defined the Temperley-Lieb algebra in 1971 in their study of planar lattice models~\cite{abramsky,bkw,temperley-lieb}, which has  been connected to the Hecke algebra, the braid group, and closely related variants such as its affine or nilpotent deformations~\cite{jones0,wenzlA,GJ}. 

The Temperley-Lieb algebra $\TL_n(\beta)$, defined for an $n\in\NN$ and parameter $\beta\in\CC$, has standard modules $W^n_\ell$ indexed by an integer $\ell\le n$ of the same parity as $n$. For $q\in\CC$ such that $\beta=q^{1/2}+q^{-1/2}$ the algebra$\TL_n(\beta)$ is a quotient of the Hecke algebra $\hecke_n(q)$, and $W_\ell^n$ pulls back to the Specht module $S^{((n+\ell)/2,(n-\ell)/2)}$. When $q$ is not a root of unity, Westbury proved by computing Gram determinants that the standard modules of $\TL_n(\beta)$ are irreducible and hence that $\TL_n(\beta)$ is semisimple~\cite{westbury}. When $q$ is a root of unity, as is the case for $\beta\in\{0,1,\sqrt2,\frac{1+\sqrt5}2,\sqrt3,2\}$, the Temperley-Lieb algebra may cease to be semisimple. Goodman and Wenzl applied the algebraic methods of evaluation at critical parameter values and spectral analysis for idempotents to obtain the block decomposition and computed the dimensions of the irreducible modules of $\TL_n(\beta)$~\cite{FWENZL}. 

\subsection{Main results}
In this paper, we focus on the specialization to $\beta=0$ corresponding to $q=-1$. For this value of $\beta$, Ridout and Saint-Aubin~\cite{RSA} computed Gram determinants to show that, the standard modules are irreducible for odd $n$, but have length two for even $n$. 

Our first main result, analogous to \cite[Theorem 1.3]{pavel} uses Ridout and Saint-Aubin's analysis of the projective modules $P_\ell^n$ of $\TL_n(0)$ to give an explicit description of the category of modules of $\TL_n(0)$. 
\begin{thm}\label{thm10}
There exists an ideal $J$ of the path algebra $\CC\quiv_{n/2}$  (defined in Definition~\ref{defn:straight-line}) for which the functor $\mathbf\Phi\colon\repcat(\TL_n(0))\to\repcat(\CC\quiv_{n/2}/J)$ given by
\[\mathbf{\Phi}(X)=\!\begin{tikzcd}
	{\hom(P_2^n,X)} & {\hom(P_4^n,X)} & {\hom(P_6^n,X)} & \cdots & {\hom(P_n^n,X)}
	\arrow["{}", shift left, from=1-1, to=1-2]
	\arrow["{}", shift left, from=1-2, to=1-1]
	\arrow["{}", shift left, from=1-2, to=1-3]
	\arrow["{}", shift left, from=1-3, to=1-2]
	\arrow["{}", shift left, from=1-3, to=1-4]
	\arrow["{}", shift left, from=1-4, to=1-3]
	\arrow["{}", shift left, from=1-4, to=1-5]
	\arrow["{}", shift left, from=1-5, to=1-4]
\end{tikzcd}\]
is an equivalence of highest weight categories $\repcat(\TL_n(0))\simeq\repcat(\CC\quiv_{n/2}/J)$. 
\end{thm}
The category equivalence implies the standard modules of $\TL_n(0)$ form an exact sequence, also observed in \cite[Theorem~A]{BNS}. 
\begin{thm}\label{thm11}
Let $n$ be even. For nonnegative even $\ell$, there exist homomorphisms $\phi^n_\ell\colon W_{\ell+2}^n\to W_\ell^n$ (defined in Definition~\ref{defn:phi}) such that the sequence \begin{equation}\label{eq:thm-exact-seq}0\longrightarrow W_n^n\xrightarrow{\phi_{n-2}^n} W_{n-2}^n\xrightarrow{\phi_{n-4}^n}\cdots\xrightarrow{\phi_2^n} W_2^n\xrightarrow{\phi_0^n} W_0^n\longrightarrow 0\end{equation} is exact. Moreover, the collection $\{\im\phi^n_\ell\mid 0\le\ell\le n-2,\ell\equiv0\pmod2\}$ are the complete set of irreducible modules of $\TL_n(0)$.
\end{thm}
The maps $\phi_\ell^n$ have explicit diagrammatic descriptions, thereby giving explicit descriptions of all irreducible modules of $\TL_n(0)$.

Since $\TL_n(0)$ is a quotient of $\hecke_n(-1)$ the exact sequence in Theorem~\ref{thm11} is also an exact sequence of $(-1)$-Specht modules. Over $\FF_2$ they give rise to an exact sequence of Specht modules of the symmetric group $\sym_n$, similarly observed in \cite[Theorem~B]{BNS},
\begin{equation}\label{eq:char2-symm-group-exact-seq}
0\longrightarrow S^{(n)}\longrightarrow S^{(n-1,1)}\longrightarrow\dots\longrightarrow S^{(n/2+1,n/2-1)}\longrightarrow S^{(n/2,n/2)}\longrightarrow0.
\end{equation}
In Corollary~\ref{exactness-3} we explicitly describe these homomorphisms.

\subsection{Relations to the Jones polynomial}\label{subsec:jones-poly}
Let $\pi\colon B_n\to\hecke_n(q)$ be the natural homomorphism, and let $\chi_\lambda$ be the character of the $q$-Specht module indexed by the partition $\lambda\vdash n$. 
Then the Jones polynomial of the closure of any braid $\alpha\in B_n$ is given by 
\[
V_{\hat\alpha}(t)=\frac{(-\sqrt{t})^{e(\alpha)-n+1}}{1+t}\sum_{k=0}^{\lfloor n/2\rfloor}\left(\sum_{i=k}^{n-k}t^i\right)\chi_{(n-k,k)}(\pi(\alpha)).
\]
Since $V_{\hat\alpha}(t)$ has no poles, the numerator must vanish at $t=-1$. If $n$ is odd the numerator always vanishes since  $\sum_{i=k}^{n-k}(-1)^i=0$, but if $n$ is even, the sum $\sum_{i=k}^{n-k}(-1)^i$ does not vanish, and we expect the identity \[\sum_{k=0}^{n/2}(-1)^k\chi_{(n-k,k)}(\pi(\alpha))=0.\]
The exact sequence~\eqref{eq:thm-exact-seq} categorifies this identity.

\subsection{Outline of the paper}
In Section~\ref{bg2}, we review preliminaries such as the Temperley-Lieb algebra, the Hecke algebra, and quiver representations. In Section~\ref{sec.3}, we compute homomorphism spaces to prove Theorem~\ref{thm10}. In Section~\ref{exactseq}, we construct the explicit diagrammatic homomorphisms between standard modules by way of proving Theorem~\ref{thm11} and illustrate the irreducible modules of $\TL_n(0)$. In Section~\ref{56}, we work with Specht  modules over $\FF_2$ to prove the exact sequence~\eqref{eq:char2-symm-group-exact-seq}. 

\subsection*{Acknowledgments.} The authors greatly appreciate the MIT PRIMES program for making this research opportunity possible. The authors thank Pavel Etingof for explaining his work with Berest and Ginzburg to us. The second author thanks Ivan Losev for patiently explaining the theory of highest weight categories and pointing out the reference \cite{BNS}. We thank Jos\'e Simental for explaining the constructions in \cite{BNS} to us.

\section{Preliminaries}\label{bg2}
\subsection{Definition of the Temperley-Lieb algebra}

Let $n$ be a positive integer.

\begin{defn}\label{tldefn}
The \textit{Temperley-Lieb algebra} $\TL_n(\beta)$ at a parameter $\beta\in\CC$ is the algebra generated by $e_1,e_2,\dots,e_{n-1}$ under the relations \[e_i^2=\beta e_i,\ e_ie_{i\pm 1}e_i=e_i,\text{ and }\ e_ie_j=e_je_i\text{ if } |i-j|\ge2.\] Its dimension is the $n$th Catalan number $\frac1{n+1}\binom{2n}n$.
\end{defn}

$\TL_n(\beta)$ has a description in terms of diagrams of strings, which consist of:
\begin{itemize}
\item a pair of horizontal lines,
\item a collection of marked points on the horizontal lines, and,
\item a collection of curves with endpoints being marked points such that no two curves intersect, and also that each marked point lies on exactly one curve.
\end{itemize}
Then, each generator is
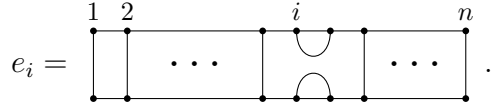
\[e_i=
\raisebox{-0.4cm}{\begin{tikzpicture}[scale=0.45]
\draw (0,0) -- (11,0);
\draw (0,2) -- (11,2);

\draw (6,0) .. controls (6,1) and (7,1) .. (7,0);
\draw (6,2) .. controls (6,1) and (7,1) .. (7,2);

\foreach \i in {0,1,5,6,7,8,11} {
\node at (\i,0) [circle,fill,inner sep=0.9pt] {};
\node at (\i,2) [circle,fill,inner sep=0.9pt] {};
}

\foreach \i in {0,1,5,8,11} {
\draw (\i,0) -- (\i,2);
}

\node[above] at (0,2) {\footnotesize $1$};
\node[above] at (1,2) {\footnotesize $2$};
\node[above] at (6,2) {\footnotesize $i$};
\node[above] at (11,2) {\footnotesize $n$};

\node at (3.1,1) {\Large $\cdots$};
\node at (9.6,1) {\Large $\cdots$};
\end{tikzpicture}}.
\]
Multiplication of basis elements amounts to the concatenation of their respective diagrams, in which the bottom of the first diagram is identified with the top of the second, and closed loops are factored out as the scalar $\beta$. 

\begin{ex}
The product $e_1e_3e_2e_1e_3\in\TL_5(\beta)$ can be computed diagrammatically as
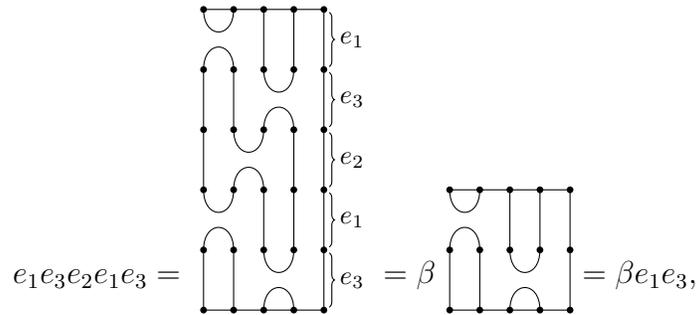
\[e_1e_3e_2e_1e_3=
\raisebox{-0.4cm}{\begin{tikzpicture}[scale=0.4]
\draw (0,0) -- (4,0);
\draw (0,10) -- (4,10);

\foreach \i in {0,...,4} {
\foreach \j in {0,...,5} {
\node at (\i,2*\j) [circle,fill,inner sep=0.9pt] {};
}
}

\foreach \i/\j in {0/0,1/0,4/0,2/1,3/1,4/1,0/2,3/2,4/2,0/3,1/3,4/3,2/4,3/4,4/4} {
\draw (\i,2*\j) -- (\i,2*\j+2);
}

\foreach \i/\j in {0/1,0/4,1/2,2/0,2/3} {
\draw (\i,2*\j) .. controls (\i,2*\j+1) and (\i+1,2*\j+1) .. (\i+1,2*\j);
\draw (\i,2*\j+2) .. controls (\i,2*\j+1) and (\i+1,2*\j+1) .. (\i+1,2*\j+2);
}

\node[above] at (0,2) {\footnotesize \phantom{$.$}};

\draw[decorate,decoration={brace,mirror,amplitude=2pt,raise=2pt}] (4,0.1) -- (4,1.9) node[midway,right=2pt] {\footnotesize $e_3$};
\draw[decorate,decoration={brace,mirror,amplitude=2pt,raise=2pt}] (4,2.1) -- (4,3.9) node[midway,right=2pt] {\footnotesize $e_1$};
\draw[decorate,decoration={brace,mirror,amplitude=2pt,raise=2pt}] (4,4.1) -- (4,5.9) node[midway,right=2pt] {\footnotesize $e_2$};
\draw[decorate,decoration={brace,mirror,amplitude=2pt,raise=2pt}] (4,6.1) -- (4,7.9) node[midway,right=2pt] {\footnotesize $e_3$};
\draw[decorate,decoration={brace,mirror,amplitude=2pt,raise=2pt}] (4,8.1) -- (4,9.9) node[midway,right=2pt] {\footnotesize $e_1$};
\end{tikzpicture}}
=\beta
\raisebox{-0.4cm}{\begin{tikzpicture}[scale=0.4]
\draw (0,0) -- (4,0);
\draw (0,4) -- (4,4);

\foreach \i in {0,...,4} {
\foreach \j in {0,...,2} {
\node at (\i,2*\j) [circle,fill,inner sep=0.9pt] {};
}
}

\foreach \i/\j in {0/0,1/0,4/0,2/1,3/1,4/1} {
\draw (\i,2*\j) -- (\i,2*\j+2);
}

\foreach \i/\j in {0/1,2/0} {
\draw (\i,2*\j) .. controls (\i,2*\j+1) and (\i+1,2*\j+1) .. (\i+1,2*\j);
\draw (\i,2*\j+2) .. controls (\i,2*\j+1) and (\i+1,2*\j+1) .. (\i+1,2*\j+2);
}

\node[above] at (0,2) {\footnotesize \phantom{$.$}};
\end{tikzpicture}}
=\beta e_1e_3,
\]
as we factor out the closed loop in the string diagram for $e_1e_3e_2e_1e_3$.
\end{ex}

\begin{defn}\label{23}
In a string diagram, we refer to curves connecting two points in its top row as \textit{cups} and curves connecting two points in its bottom row as \textit{caps}. Curves connecting the top and bottom rows are referred to as \textit{throughlines}.
\end{defn}

\begin{ex}
The element\[e_1e_3=
\raisebox{-0.4cm}
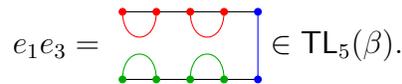
{\begin{tikzpicture}[scale=0.45]
\draw (0,0) -- (4,0);
\draw (0,2) -- (4,2);

\foreach \i in {0,...,3} {
\node[draw=green!67!black,fill=green!67!black] at (\i,0) [circle,fill,inner sep=0.9pt] {};
\node[draw=red,fill=red] at (\i,2) [circle,fill,inner sep=0.9pt] {};
}

\foreach \i in {0,2} {
\node[draw=blue,fill=blue] at (4,\i) [circle,fill,inner sep=0.9pt] {};
}

\draw[draw=blue] (4,0) -- (4,2);

\foreach \i/\j in {0/0,2/0} {
\draw[draw=green!67!black] (\i,2*\j) .. controls (\i,2*\j+1) and (\i+1,2*\j+1) .. (\i+1,2*\j);
\draw[draw=red] (\i,2*\j+2) .. controls (\i,2*\j+1) and (\i+1,2*\j+1) .. (\i+1,2*\j+2);
}

\node[above] at (0,2) {\footnotesize \phantom{$.$}};
\end{tikzpicture}}
\in\TL_5(\beta).\]
has two \textcolor{red}{cups}, two \textcolor{green!67!black}{caps}, and one \textcolor{blue}{throughline}.
\end{ex}

Now let $\ell\le n$ be a nonnegative integer such that $\ell\equiv n\pmod 2$.
\begin{defn}
A diagram of strings from $n$ points above to $\ell$ points below is \textit{monic} if there are no caps. The $\CC$-vector space spanned by the basis of all monic diagrams forms the \textit{standard module} $W_\ell^n$, which has dimension $\binom{n}{\frac{n-\ell}2}-\binom{n}{\frac{n-\ell}2-1}$. The standard modules are naturally acted upon by $\TL_n(\beta)$ via concatenation of diagrams, with any resultant non-monic diagram due to the formation of caps defined to be equal to $0$.
\end{defn}

\begin{ex}\label{ex2}
If
$x=
\raisebox{-0.4cm}{\begin{tikzpicture}[scale=0.45]
\draw (0,0) -- (5,0);
\draw (0,2) -- (5,2);

\node at (0,0) [circle,fill,inner sep=0.9pt] {};
\node at (1,0) [circle,fill,inner sep=0.9pt] {};

\foreach \i in {0,...,5} {
\node at (\i,2) [circle,fill,inner sep=0.9pt] {};
}

\foreach \i in {0,1} {
\draw (\i,0) -- (\i,2);
}

\foreach \i in {2,4} {
\draw (\i,2) .. controls (\i,1) and (\i+1,1) .. (\i+1,2);
}

\node[above] at (0,2) {\footnotesize \phantom{$.$}};
\end{tikzpicture}}\in W_2^6$
then
\[e_3x=
\raisebox{-0.4cm}{\begin{tikzpicture}[scale=0.45]
\draw (0,0) -- (5,0);
\draw (0,4) -- (5,4);

\node at (0,0) [circle,fill,inner sep=0.9pt] {};
\node at (1,0) [circle,fill,inner sep=0.9pt] {};

\foreach \i in {0,...,5} {
\node at (\i,2) [circle,fill,inner sep=0.9pt] {};
\node at (\i,4) [circle,fill,inner sep=0.9pt] {};
}

\foreach \i in {0,1} {
\draw (\i,0) -- (\i,2);
}

\foreach \i in {0,1,2,5} {
\draw (\i,2) -- (\i,4);
}

\foreach \i in {2,4} {
\draw (\i,2) .. controls (\i,1) and (\i+1,1) .. (\i+1,2);
}

\draw (3,2) .. controls (3,3) and (4,3) .. (4,2);
\draw (3,4) .. controls (3,3) and (4,3) .. (4,4);

\node[above] at (0,2) {\footnotesize \phantom{$.$}};
\end{tikzpicture}}
=
\raisebox{-0.4cm}{\begin{tikzpicture}[scale=0.45]
\draw (0,0) -- (5,0);
\draw (0,2) -- (5,2);

\node at (0,0) [circle,fill,inner sep=0.9pt] {};
\node at (1,0) [circle,fill,inner sep=0.9pt] {};

\foreach \i in {0,...,5} {
\node at (\i,2) [circle,fill,inner sep=0.9pt] {};
}

\foreach \i in {0,1} {
\draw (\i,0) -- (\i,2);
}

\draw (2,2) .. controls (2,0) and (5,0) .. (5,2);
\draw (3,2) .. controls (3,1) and (4,1) .. (4,2);

\node[above] at (0,2) {\footnotesize \phantom{$.$}};
\end{tikzpicture}}
\]
is monic and is thus another basis element of $W^6_2$.

On the other hand,
\[e_1x=
\raisebox{-0.4cm}{\begin{tikzpicture}[scale=0.45]
\draw (0,0) -- (5,0);
\draw (0,4) -- (5,4);

\node at (0,0) [circle,fill,inner sep=0.9pt] {};
\node at (1,0) [circle,fill,inner sep=0.9pt] {};

\foreach \i in {0,...,5} {
\node at (\i,2) [circle,fill,inner sep=0.9pt] {};
\node at (\i,4) [circle,fill,inner sep=0.9pt] {};
}

\foreach \i in {0,1} {
\draw (\i,0) -- (\i,2);
}

\foreach \i in {2,3,4,5} {
\draw (\i,2) -- (\i,4);
}

\foreach \i in {2,4} {
\draw (\i,2) .. controls (\i,1) and (\i+1,1) .. (\i+1,2);
}

\draw (0,2) .. controls (0,3) and (1,3) .. (1,2);
\draw (0,4) .. controls (0,3) and (1,3) .. (1,4);

\node[above] at (0,2) {\footnotesize \phantom{$.$}};
\end{tikzpicture}}
=
\raisebox{-0.4cm}{\begin{tikzpicture}[scale=0.45]
\draw (0,0) -- (5,0);
\draw (0,2) -- (5,2);

\node at (0,0) [circle,fill,inner sep=0.9pt] {};
\node at (1,0) [circle,fill,inner sep=0.9pt] {};

\foreach \i in {0,...,5} {
\node at (\i,2) [circle,fill,inner sep=0.9pt] {};
}

\draw (0,0) .. controls (0,1) and (1,1) .. (1,0);
\draw (0,2) .. controls (0,1) and (1,1) .. (1,2);
\draw (2,2) .. controls (2,1) and (3,1) .. (3,2);
\draw (4,2) .. controls (4,1) and (5,1) .. (5,2);

\node[above] at (0,2) {\footnotesize \phantom{$.$}};
\end{tikzpicture}}
\]
is not monic, so $e_1x=0$.
\end{ex}

As discussed in Section~\ref{intro2}, the standard modules are irreducible for generic values of $\beta$.

\begin{defn}\label{bndefn}
The \textit{braid group} $B_n$ is generated by $\sigma_1,\sigma_2,\dots,\sigma_{n-1}$ with relations \[\sigma_i\sigma_{i+1}\sigma_i=\sigma_{i+1}\sigma_i\sigma_{i+1}\ \text{ and }\ \sigma_i\sigma_j=\sigma_j\sigma_i\text{ if }|i-j|\ge2.\] Each braid can be seen as $n$ intersecting strands of string, in which each $\sigma_i$ introduces a twist on the strands in the $i$th and $(i+1)$th positions.
\end{defn}

\begin{defn}\label{heckedefn}
The \textit{Hecke algebra} $\hecke_n(q)$ at a parameter $q\in\CC\setminus\{0\}$ is the algebra generated by $g_1,g_2,\dots,g_{n-1}$ with relations
\[(g_i-q)(g_i+1)=0,\ g_ig_{i+1}g_i=g_{i+1}g_ig_{i+1},\text{ and }\ g_ig_j=g_jg_i\text{ if }|i-j|\ge2.\]
\end{defn}

When $\beta=q^{1/2}+q^{-1/2}$, the Temperley-Lieb algebra $\TL_n(\beta)$ is a quotient of $\hecke_n(q)$.

\begin{prop}\label{surj}
Let $\beta,q\in\CC$ such that $\beta=q^{1/2}+q^{-1/2}$ and $q\neq0$. Then the homomorphism $\theta\colon\hecke_n(q)\to\TL_n(\beta)$ where $\theta(g_i)=q^{1/2}e_i-1$ is surjective.
\end{prop}

\subsection{Quivers} We briefly review quivers and their representations.

\begin{defn}
A \textit{quiver} $\quiv$ is a directed graph in which loops and multiple edges are allowed. A \textit{path} in $\quiv$ is defined in the familiar graph-theoretic manner, with vertices and edges permitted to appear multiple times. Trivial paths of length zero are also allowed. For every path $p$ of $\quiv$, let $s(p)$ and $t(p)$ denote the starting and terminal vertices of $p$.

Given two paths $p$ and $q$ such that $s(p)=t(q)$, define $p\circ q$ to be the path that starts at $s(q)$, traverses along $q$ to reach $s(p)=t(q)$, and then traverses along $p$ to terminate at $t(p)$.
\end{defn}
\begin{defn}\label{repquiver}
Let $\quiv$ be a quiver. The \textit{path algebra} $\CC\quiv$ of $\quiv$ is the $\CC$-vector space spanned by all paths on $\quiv$ such that, for paths $p$ and $q$ of $\quiv$, we have \[pq=\begin{cases}p\circ q&\text{if $s(p)=t(q)$}\\0&\text{otherwise.}\end{cases}\] A \textit{representation of $\quiv$} is a collection of vector spaces and maps endowed with a bijection assigning each vertex of $\quiv$ to a vector space and each directed edge $e$ of $\quiv$ to a map between the vector spaces associated with $s(e)$ and $t(e)$. Then representations of $\quiv$ are equivalent to $\CC\quiv$-modules. Let $\repcat(\CC\quiv)$ denote the category of $\CC\quiv$-modules.
\end{defn}

\subsection{Highest weight categories}
Recall from \cite{HW-categories} the notion of a \emph{highest weight category}.
\begin{defn}\label{def:hw-category}
    Let $\mathcal O$ be a $\CC$-linear artinian abelian category, and let $(\Lambda,\preceq)$ be a poset labeling the simple objects $L(\lambda)$ of $\mathcal O$. Let $P(\lambda)$ be the projective cover of $L(\lambda)$. A \emph{highest weight structure} on $\mathcal O$ is a set of standard objects $\{\Delta(\lambda)\mid \lambda\in\Lambda\}$, such that
    \begin{itemize}
        \item if $\hom(\Delta(\lambda),\Delta(\mu))\ne0$ then $\lambda\preceq \mu$,
        \item $\en(\Delta(\lambda))=\CC$, and
        \item there is an epimorphism $P(\lambda)\twoheadrightarrow\Delta(\lambda)$ whose kernel is filtered with quotients of the form $\Delta(\mu)$ for $\mu\succ\lambda$.
    \end{itemize}
\end{defn}

\section{A description of the category of representations of the Temperley-Lieb algebra}\label{sec.3}
\subsection{Highest weight structure on representations of the quiver}
We first introduce the straight-line quiver.
\begin{defn}\label{defn:straight-line}
The \textit{straight-line quiver} $\quiv_m$ is the quiver on $m$ vertices has the structure
\[\begin{tikzcd}
	\bullet & \bullet & \bullet & \cdots & \bullet.
	\arrow["{a_1}", shift left, from=1-1, to=1-2]
	\arrow["{b_1}", shift left, from=1-2, to=1-1]
	\arrow["{a_2}", shift left, from=1-2, to=1-3]
	\arrow["{b_2}", shift left, from=1-3, to=1-2]
	\arrow["{a_3}", shift left, from=1-3, to=1-4]
	\arrow["{b_3}", shift left, from=1-4, to=1-3]
	\arrow["{a_{m-1}}", shift left, from=1-4, to=1-5]
	\arrow["{b_{m-1}}", shift left, from=1-5, to=1-4]
\end{tikzcd}\]
We let $e_i$ denote the trivial path on the $i$th leftmost vertex. Define the ideal \begin{equation}\label{eq:defn-of-J} J=(a_{i+1}a_i,b_ib_{i+1},a_ib_i-b_{i+1}a_{i+1}\mid 1\le i\le m)\subset\CC\quiv_m. \end{equation}
\end{defn}

\begin{prop}\label{311}
The category $\repcat(\CC\quiv_m/J)$ exhibits a highest weight structure with poset $\{1,2,\dots,m\}^{\mathrm{op}}$. It has simple objects
\[L(i)\coloneq\begin{tikzcd}
	\cdots & 0 & {\mathbb{C}} & 0 & {\cdots,}
	\arrow[shift left, from=1-1, to=1-2]
	\arrow[shift left, from=1-2, to=1-1]
	\arrow[shift left, from=1-2, to=1-3]
	\arrow[shift left, from=1-3, to=1-2]
	\arrow[shift left, from=1-3, to=1-4]
	\arrow[shift left, from=1-4, to=1-3]
	\arrow[shift left, from=1-4, to=1-5]
	\arrow[shift left, from=1-5, to=1-4]
\end{tikzcd}\]
standard objects
\[\Delta(i)\coloneq\begin{tikzcd}
	\cdots & 0 & {\mathbb{C}} & {\mathbb{C}} & 0 & {\cdots,}
	\arrow[shift left, from=1-1, to=1-2]
	\arrow[shift left, from=1-2, to=1-1]
	\arrow[shift left, from=1-2, to=1-3]
	\arrow[shift left, from=1-3, to=1-2]
	\arrow["{\operatorname{id}}", shift left, from=1-3, to=1-4]
	\arrow["0", shift left, from=1-4, to=1-3]
	\arrow[shift left, from=1-4, to=1-5]
	\arrow[shift left, from=1-5, to=1-4]
	\arrow[shift left, from=1-5, to=1-6]
	\arrow[shift left, from=1-6, to=1-5]
\end{tikzcd}\]
and projective objects
\[P(i)\coloneq\begin{tikzcd}
	\cdots & 0 & {\mathbb{C}} & {\mathbb{C}^2} & {\mathbb{C}} & 0 & {\cdots,}
	\arrow[shift left, from=1-1, to=1-2]
	\arrow[shift left, from=1-2, to=1-1]
	\arrow[shift left, from=1-2, to=1-3]
	\arrow[shift left, from=1-3, to=1-2]
	\arrow["{\iota_2}", shift left, from=1-3, to=1-4]
	\arrow["{\pi_1}", shift left, from=1-4, to=1-3]
	\arrow["{\pi_1}", shift left, from=1-4, to=1-5]
	\arrow["{\iota_2}", shift left, from=1-5, to=1-4]
	\arrow[shift left, from=1-5, to=1-6]
	\arrow[shift left, from=1-6, to=1-5]
	\arrow[shift left, from=1-6, to=1-7]
	\arrow[shift left, from=1-7, to=1-6]
\end{tikzcd}\]
where $\operatorname{id}$ is the identity, $\pi_1$ is the projection into the first component, and $\iota_2$ is the inclusion onto the second component. Indexing is determined as follows: for each $L(i)$ (resp. $P(i)$), attach the space $\CC$ (resp. $\CC^2$) to the $i$th leftmost vertex of $\quiv_m$. For $\Delta(i)$, attach the leftmost copy of $\CC$ to the $i$th leftmost vertex of $\quiv_m$.
\end{prop}
\begin{proof}
Since $P(i)\cong(\CC\quiv_m/J)e_i$ where $e_i$ is idempotent, the module $P(i)$ is projective with basis $\{e_i,b_ia_ie_i,a_ie_i,b_{i-1}e_i\}$. Now $L(i)$ is one-dimensional and spanned by $v_i$, with action of $\CC\quiv_m/J$ given by $e_iv_i=v_i$ and $a_iv_i=b_{i-1}v_i=0$. Let $\xi\colon P(i)\to L(i)$ be the epimorphism satisfying $\xi(e_i)=v_i$ with kernel spanned by $\{b_ia_ie_i,a_ie_i,b_{i-1}e_i\}$.

We check that $\xi$ is a projective cover. Take any submodule $N\subset P(i)$ whose image under $\xi$ equals $L(i)$, so there exists $p'\in\ker\xi$ such that $e_i+p'\in N$. Then since $\ker\xi$ is annihilated by $a_i$ and $b_{i-1}$
\[a_i(e_i+p')=a_ie_i,\ b_i(e_i+p')=b_ie_i,\ b_ia_ie_i\in N,\]
so $\ker\xi\subset N$ and $e_i=(e_i+p')-p'\in N$. Thus $N=P(i)$ and $\xi$ is an essential surjection.

One readily checks that $L(i)$, $\Delta(i)$, and $P(i)$ satisfy the first two axioms of Definition~\ref{def:hw-category}. The third axiom follows from the exact sequence $\Delta(i-1)\to P(i)\to\Delta(i)$ written out as
\[\begin{tikzcd}
	\cdots & 0 & {\mathbb{C}} & {\mathbb{C}} & 0 & 0 & {\cdots} \\
	\cdots & 0 & {\mathbb{C}} & {\mathbb{C}^2} & {\mathbb{C}} & 0 & {\cdots} \\
	\cdots & 0 & 0 & {\mathbb{C}} & {\mathbb{C}} & 0 & {\cdots.}
	\arrow[shift left, from=1-1, to=1-2]
	\arrow[shift left, from=1-2, to=1-1]
	\arrow[shift left, from=1-2, to=1-3]
	\arrow[shift left, from=1-3, to=1-2]
	\arrow[shift left, from=1-3, to=1-4]
	\arrow["{\operatorname{id}}", from=1-3, to=2-3]
	\arrow[shift left, from=1-4, to=1-3]
	\arrow[shift left, from=1-4, to=1-5]
	\arrow["{\iota_2}", from=1-4, to=2-4]
	\arrow[shift left, from=1-5, to=1-4]
	\arrow[shift left, from=1-5, to=1-6]
	\arrow["0", from=1-5, to=2-5]
	\arrow[shift left, from=1-6, to=1-5]
	\arrow[shift left, from=1-6, to=1-7]
	\arrow[shift left, from=1-7, to=1-6]
	\arrow[shift left, from=2-1, to=2-2]
	\arrow[shift left, from=2-2, to=2-1]
	\arrow[shift left, from=2-2, to=2-3]
	\arrow[shift left, from=2-3, to=2-2]
	\arrow[shift left, from=2-3, to=2-4]
	\arrow["0", from=2-3, to=3-3]
	\arrow[shift left, from=2-4, to=2-3]
	\arrow[shift left, from=2-4, to=2-5]
	\arrow["{\pi_1}", from=2-4, to=3-4]
	\arrow[shift left, from=2-5, to=2-4]
	\arrow[shift left, from=2-5, to=2-6]
	\arrow["{\operatorname{id}}", from=2-5, to=3-5]
	\arrow[shift left, from=2-6, to=2-5]
	\arrow[shift left, from=2-6, to=2-7]
	\arrow[shift left, from=2-7, to=2-6]
	\arrow[shift left, from=3-1, to=3-2]
	\arrow[shift left, from=3-2, to=3-1]
	\arrow[shift left, from=3-2, to=3-3]
	\arrow[shift left, from=3-3, to=3-2]
	\arrow[shift left, from=3-3, to=3-4]
	\arrow[shift left, from=3-4, to=3-3]
	\arrow[shift left, from=3-4, to=3-5]
	\arrow[shift left, from=3-5, to=3-4]
	\arrow[shift left, from=3-5, to=3-6]
	\arrow[shift left, from=3-6, to=3-5]
	\arrow[shift left, from=3-6, to=3-7]
	\arrow[shift left, from=3-7, to=3-6]
\end{tikzcd}\]
\end{proof}

\begin{rem}
Let $\perv(\PP^m)$ be the category of perverse sheaves on $\PP^m$ with stratification $\PP^m=\bigcup_{i=0}^m\mathbb{A}^i$. Comparing \eqref{eq:defn-of-J} to the description of $\perv(\PP^m)$ in \cite{bonfert}, we see $\repcat(\CC\quiv_m/J)$ is equivalent to the quotient of $\perv(\PP^{m+1})$ by the subcategory generated by the constant sheaf. 
\end{rem}

\subsection{Highest weight structure on representations of the Temperley-Lieb algebra}
We now specialize the Temperley-Lieb algebra to $\beta=0$.

\begin{prop}[\RSAthreethree]\label{prop33}
Let $\ell>0$. For any basis elements $x,y\in W_\ell^n$, let $\alpha(x)$ be the string diagram obtained by reflecting $x$ horizontally. Let $\alpha(x,y)\in\TL_\ell(0)$ be the element obtained by diagrammatically concatenating $\alpha(x)$ above $y$, and let $\langle\cdot,\cdot\rangle\colon W_\ell^n\otimes W_\ell^n\to\CC$ be the pairing such that for basis elements $x,y\in W_\ell^n$,\[\langle x,y\rangle=\begin{cases}1&\text{if $\alpha(x,y)\in\TL_\ell(0)$ contains $\ell$ throughlines}\\0&\text{otherwise.}\end{cases}\] Then, the quotient modules $L_\ell^n=W^n_\ell/\{ x\in W^n_\ell\mid\langle x,y\rangle=0\quad\forall y\in W^n_\ell\}$ are irreducible.
\end{prop}

Now, fix a positive \emph{even} integer $n$.

\begin{prop}[\RSAeighttwo]\label{prop822}
The collection $\{W_{\ell-1}^{n-1}\mid 2\le\ell\le n,\ell\equiv0\pmod2\}$ of standard modules form a complete set of pairwise distinct irreducible modules of $\TL_{n-1}(0)$. Additionally, the standard module $W^{n-1}_{\ell-1}$ is also projective.
\end{prop}

\begin{cor}\label{oddsemisimple}
The algebra $\TL_{n-1}(0)$ is semisimple.
\end{cor}
\begin{proof}
Follows from Proposition~\ref{prop822} as all the irreducibles are projective.
\end{proof}

When $\beta=0$ the $W_\ell^n$ are no longer necessarily irreducible, but they are the standard modules which makes $\repcat(\TL_n(0))$ into a highest weight category. This is a general fact, but we check this directly using results of Ridout and Saint-Aubin \cite{RSA}.

\begin{prop}[\RSAsevenfour]\label{prop74}
The collection $\{L_\ell^n\mid 2\le\ell\le n,\ell\equiv0\pmod2\}$ of quotient modules form a complete set of distinct irreducible modules of $\TL_n(0)$. Moreover, the sequence \[0\longrightarrow L_{\ell+2}^n\longrightarrow W_\ell^n \longrightarrow L_\ell^n \longrightarrow0\] is exact and non-split for each $\ell$.
\end{prop}

These irreducible objects admit a projective cover.
\begin{prop}[\RSAeighttwo]\label{prop821}
For each $\ell>0$ the modules \[P^n_\ell:=\ind_{\TL_{n-1}(0)}^{\TL_n(0)}W_{\ell-1}^{n-1}\] form a projective cover of $L_\ell^n$. They sit in a short exact sequence 
\begin{equation}\label{eq:P-SES}
0\longrightarrow W_{\ell-2}^n\longrightarrow P_\ell^n \longrightarrow W_\ell^n \longrightarrow0.
\end{equation}
\end{prop}
\begin{rem}
The quotient and projective modules admit diagrammatic descriptions. From Proposition~\ref{prop821}, the basis of $P_\ell^n=\ind_{\TL_{n-1}(0)}^{\TL_n(0)}W_{\ell-1}^{n-1}$ can be interpreted as consisting of diagrams of strings from $n$ points above to $\ell$ points below in which the only cap permitted connects the rightmost two points on the bottom.

The diagrammatic description of $L^n_\ell$ is more subtle, which we give in Section~\ref{exactseq}.
\end{rem}

\begin{prop}[\RSAfourtwo]\label{prop42}
There exists an isomorphism of $\TL_n(0)$-modules
\[W^n_\ell|_{\TL_{n-1}(0)}\cong W^{n-1}_{\ell-1}\oplus W^{n-1}_{\ell+1}.\]
\end{prop}

Using the above, we deduce some new results.
\begin{prop}\label{homPW}
Let $\ell$ and $m$ be positive even integers no greater than $n$. Then
\begin{align*}
\dim\hom_{\TL_n(0)}(P_\ell^n,W_m^n)=\begin{cases}1&\text{if $\ell\in\{m,m+2\}$}\\0&\text{otherwise.}\end{cases}
\end{align*}
\end{prop}
\begin{proof}
By Proposition~\ref{prop821} and Frobenius reciprocity \[\hom_{\TL_n(0)}(P_\ell^n,W_m^n)=\hom_{\TL_{n-1}(0)}(W^{n-1}_{\ell-1},W_m^n|_{\TL_{n-1}(0)}).\] Hence by Proposition~\ref{prop42}, we find that \[\hom_{\TL_n(0)}(P_\ell^n,W_m^n)=\hom_{\TL_n(0)}(W^{n-1}_{\ell-1},W^{n-1}_{m-1}\oplus W^{n-1}_{m+1}).\] 
Since $W^{n-1}_{\ell-1}$, $W^{n-1}_{m-1}$, and $W^{n-1}_{m+1}$ are irreducible by Proposition~\ref{prop822}, the result follows.
\end{proof}

\begin{thm}\label{homP}
Let $\ell$ and $m$ be positive even integers no greater than $n$. Then
\begin{align*}
\dim\hom_{\TL_n(0)}(P_\ell^n,P_m^n)=\begin{cases}2&\text{if $\ell=m$}\\1&\text{if $|\ell-m|=2$}\\0&\text{otherwise.}\end{cases}
\end{align*}
\end{thm}
\begin{proof}
By Frobenius reciprocity, \[\hom_{\TL_n(0)}(P_\ell^n,P_m^n)=\hom_{\TL_n(0)}(\ind_{\TL_{n-1}(0)}^{\TL_n(0)}W^{n-1}_{\ell-1},P_m^n)=\hom_{\TL_{n-1}(0)}(W^{n-1}_{\ell-1},P_m^n|_{\TL_{n-1}(0)}).\]
By restricting the exact sequence of Proposition~\ref{prop821} to $\TL_{n-1}(0)\subset\TL_n(0)$ and applying Proposition~\ref{prop42} we see
\[0\longrightarrow W_{m-3}^{n-1}\oplus W_{m-1}^{n-1}\longrightarrow P_m^n|_{\TL_{n-1}(0)} \longrightarrow W_{m-1}^{n-1}\oplus W_{m+1}^{n-1}\longrightarrow0\] is exact. By Corollary~\ref{oddsemisimple} the sequence splits so
\begin{equation}\label{resPnm}
P_m^n|_{\TL_{n-1}(0)}\cong W_{m-3}^{n-1}\oplus W_{m-1}^{n-1}\oplus W_{m-1}^{n-1}\oplus W_{m+1}^{n-1},
\end{equation}
and the result follows.
\end{proof}

The category $\repcat(\TL_n(0))$ carries a highest weight structure.

\begin{cor}\label{cor:hw-structure-on-TL}
    The category $\repcat(\TL_n(0))$ with the simple objects $L_\ell^n$ and standard objects $W_\ell^n$ forms a highest weight category with respect to $\{2,\dots,n\}^{\mathrm{op}}$.
\end{cor}
\begin{proof}
    We check each axiom of Definition~\ref{def:hw-category} directly. By \eqref{eq:P-SES} we see $\hom(W_\ell^n,W_m^n)\subset\hom(P_\ell^n,W_m^n)$, so by Proposition~\ref{homPW} we conclude that if $\hom(W_\ell^n,W_m^n)\ne0$ then $\ell\ge m$. Moreover when $\ell=m$ we see $\hom(W_\ell^n,W_\ell^n)$ must be one-dimensional. Finally \eqref{eq:P-SES} gives a standard filtration on $P_\ell^n$.
\end{proof}
\subsection{A category equivalence}
Let $\omega_\ell^n\colon P_\ell^n\to P_{\ell+2}^n$ and $\gamma_\ell^n\colon P_{\ell+2}^n\to P_\ell^n$ be nonzero maps between adjacent projectives as in Theorem~\ref{homP}. 
We organize them into the diagram
\[\begin{tikzcd}
	P_2^n & P_4^n & P_6^n & \cdots & P_{n}^n.
	\arrow["{\omega_2^n}", shift left, from=1-1, to=1-2]
	\arrow["{\gamma_2^n}", shift left, from=1-2, to=1-1]
	\arrow["{\omega_4^n}", shift left, from=1-2, to=1-3]
	\arrow["{\gamma_4^n}", shift left, from=1-3, to=1-2]
	\arrow["{\omega_6^n}", shift left, from=1-3, to=1-4]
	\arrow["{\gamma_6^n}", shift left, from=1-4, to=1-3]
	\arrow["{\omega_{n-2}^n}", shift left, from=1-4, to=1-5]
	\arrow["{\gamma_{n-2}^n}", shift left, from=1-5, to=1-4]
\end{tikzcd}\]
We check that the morphisms $\omega_\ell^n$ and $\gamma_\ell^n$ satisfy the relations \eqref{eq:defn-of-J}.
\begin{lemma}\label{nonzerocomp}
For all $\ell$, the compositions $\omega_\ell^n\circ\gamma_\ell^n$ and $\gamma_{\ell+2}^n\circ\omega_{\ell+2}^n$ are nonzero and equal.
\end{lemma}
\begin{proof}
Comparing Propositions~\ref{homPW} and~\ref{homP}, observe that $\omega_\ell^n\colon P^n_\ell\to P^n_{\ell+2}$ factors through $W_\ell^n$, and so the composition $\omega_\ell^n\circ\gamma_\ell^n$ factors as $P_{\ell+2}^n\to W_\ell^n\subset P_{\ell+2}^n$, which by Frobenius reciprocity corresponds to a homomorphism \[W_{\ell+1}^{n-1}\longrightarrow P_\ell^n|_{\TL_{n-1}(0)}\longrightarrow W_\ell^n|_{\TL_{n-1}(0)}.\] This can be rewritten by Proposition~\ref{prop42} and \eqref{resPnm} as \[W_{\ell+1}^{n-1}\longrightarrow W_{\ell-3}^{n-1}\oplus W_{\ell-1}^{n-1}\oplus W_{\ell-1}^{n-1}\oplus W_{\ell+1}^{n-1}\longrightarrow W_{\ell-1}^{n-1}\oplus W_{\ell+1}^{n-1},\] and the above composition is the inclusion of $W_{\ell+1}^{n-1}$ into the second factor.

Similarly, the composition $\gamma_{\ell+2}^n\circ\omega_{\ell+2}^n$ factors as $P_{\ell+2}^n\longrightarrow W_{\ell+2}^n\longrightarrow P_{\ell+2}^n$, which by Frobenius reciprocity corresponds to\[W_{\ell+1}^{n-1}\longrightarrow W_{\ell+1}^{n-1}\oplus W_{\ell+3}^{n-1}\longrightarrow W_{\ell-1}^{n-1}\oplus W_{\ell+1}^{n-1}\oplus W_{\ell+1}^{n-1}\oplus W_{\ell+3}^{n-1}\] which is again the inclusion of $W_{\ell+1}^{n-1}$ into the second factor.
\end{proof}
We now have all the ingredients to prove Theorem~\ref{thm10}.
\begin{proof}[Proof of Theorem~\ref{thm10}]
Since $P^{\circ n}\coloneq\bigoplus_{i=1}^{n/2} P_{2i}^n$ is a projective generator of $\repcat(\TL_n(0))$, we have an equivalence
\[\hom(P^{\circ n},-)\colon\repcat(\TL_n(0))\simeq\repcat(\en(P^{\circ n})).\] By an abuse of notation, we write $\omega_\ell^n$ and $\gamma_\ell^n$ for the corresponding endomorphisms of $P^{\circ n}$. Consider the homomorphism $\Psi\colon \CC\quiv_{n/2}/J\to\en(P^{\circ n})$ where $\Psi(a_i)=\omega_{2i}^n$ and $\Psi(b_i)=\gamma_{2i}^n$. Then $\Psi(a_{i+1}a_i)=\omega_{2i+2}^n\circ\omega_{2i}^n=0$, $\Psi(b_ib_{i+1})=\gamma_{2i}^n\circ\gamma_{2i+2}^n=0$, and $\Psi(a_ib_i-b_{i+1}a_{i+1})=\omega_{2i}^n\circ\gamma_{2i}^n-\gamma_{2i+2}^n\circ\omega_{2i+2}^n=0$ by Theorem~\ref{homP} and Lemma~\ref{nonzerocomp}. It follows that $J\subset\ker\Psi$, implying that $\Psi$ is a well-defined homomorphism.

The surjectivity of $\Psi$ is clear. Thus, it still remains to check that $\CC\quiv_{n/2}/J$ and $\en(P^{\circ n})$ have equal dimension. Both $\CC\quiv_{n/2}/J$ and $\en(P^{\circ n})$ are bigraded vector spaces, so we can just check that the graded pieces have equal dimension, which is clear from simple counting. Comparing Corollary~\ref{cor:hw-structure-on-TL} and Proposition~\ref{311} implies the construction of $\Phi$.
\end{proof}

\begin{rem}
    Equivalently, we could directly apply Barr-Beck to the restriction functor $\repcat(\TL_n(0))\to\repcat(\TL_{n-1}(0))$ and use that $\TL_{n-1}(0)$ is semisimple (Corollary~\ref{oddsemisimple}).
\end{rem}

\begin{cor}\label{exactseq-abstract-nonsense}
There exists a long exact sequence on the standard modules
\begin{equation}\label{eq:exact-sequence-from-abstract}
0\longrightarrow W^n_n\longrightarrow W^n_{n-2}\longrightarrow\dots\longrightarrow W_2^n\longrightarrow W_0^n\longrightarrow 0.\end{equation}
\end{cor}
\begin{proof}
Retain the same notation from Theorem~\ref{thm10}. We can compute $\mathbf\Phi(W_\ell^n)$ for even $\ell$ by Proposition~\ref{homPW}, as
\[
\mathbf\Phi(W_\ell^n)=\begin{cases}
\begin{tikzcd}
	{\CC} & {0} & {0} & \cdots & {0,}
	\arrow["{}", shift left, from=1-1, to=1-2]
	\arrow["{}", shift left, from=1-2, to=1-1]
	\arrow["{}", shift left, from=1-2, to=1-3]
	\arrow["{}", shift left, from=1-3, to=1-2]
	\arrow["{}", shift left, from=1-3, to=1-4]
	\arrow["{}", shift left, from=1-4, to=1-3]
	\arrow["{}", shift left, from=1-4, to=1-5]
	\arrow["{}", shift left, from=1-5, to=1-4]
\end{tikzcd}&\text{if }\ell=0\\
\begin{tikzcd}
	0 & 0 & \cdots & {\mathbb{C}} & {\mathbb{C}} & \cdots & 0,
	\arrow[shift left, from=1-1, to=1-2]
	\arrow[shift left, from=1-2, to=1-1]
	\arrow[shift left, from=1-2, to=1-3]
	\arrow[shift left, from=1-3, to=1-2]
	\arrow[shift left, from=1-3, to=1-4]
	\arrow[shift left, from=1-4, to=1-3]
	\arrow[shift left, from=1-4, to=1-5]
	\arrow[shift left, from=1-5, to=1-4]
	\arrow[shift left, from=1-5, to=1-6]
	\arrow[shift left, from=1-6, to=1-5]
	\arrow[shift left, from=1-6, to=1-7]
	\arrow[shift left, from=1-7, to=1-6]
\end{tikzcd}&\text{if }\ell\notin\{0,n\}\\
\begin{tikzcd}
	{0} & {0} & {0} & \cdots & {\CC,}
	\arrow["{}", shift left, from=1-1, to=1-2]
	\arrow["{}", shift left, from=1-2, to=1-1]
	\arrow["{}", shift left, from=1-2, to=1-3]
	\arrow["{}", shift left, from=1-3, to=1-2]
	\arrow["{}", shift left, from=1-3, to=1-4]
	\arrow["{}", shift left, from=1-4, to=1-3]
	\arrow["{}", shift left, from=1-4, to=1-5]
	\arrow["{}", shift left, from=1-5, to=1-4]
\end{tikzcd}&\text{if }\ell=n.
\end{cases}
\]
Now the exact sequence follows from the obvious exact sequence
\[0\longrightarrow\mathbf{\Phi}(W^n_n)\longrightarrow\mathbf{\Phi}(W^n_{n-2})\longrightarrow\dots\longrightarrow \mathbf{\Phi}(W_2^n)\longrightarrow \mathbf{\Phi}(W_0^n)\longrightarrow 0.\qedhere\]
\end{proof}
\begin{rem}
    The exact sequence~\eqref{eq:exact-sequence-from-abstract} also has a counterpart with the rational Cherednik algebra $H_{n}(\frac12)$. In particular, \cite[Theorem~A]{BNS} proves an exact sequence
    \[
    0\to\Delta(1^n)\to\Delta(2,1^{n-2})\to\cdots\to\Delta(2^{n/2},2^{n/2})\to L(2^{n/2},2^{n/2})\to0.
    \]
    The Knizhnik-Zamolodchikov functor kills the object $L(2^{n/2},2^{n/2})$ and recovers \eqref{eq:exact-sequence-from-abstract}.
\end{rem}

\section{An exact sequence of homomorphisms on the standard modules}\label{exactseq}
In this section, we explicitly construct the homomorphisms in \eqref{eq:exact-sequence-from-abstract}.

\begin{defn}\label{defn:phi}
For any monic basis element $x\in W_{\ell+2}^n$, let the diagram $x\delta_i^\ell\in W_{\ell}^n$ connect the $(i+1)$th and $(i+2)$th lower leftmost points of $x$ with a cup. Let $\phi_\ell^n\colon W_{\ell+2}^n\to W_\ell^n$ be the alternating sum \[\phi_\ell^n(x)=\sum_{i=0}^{\ell/2}(-1)^ix\delta_{2i}^n.\]
\end{defn}

\begin{ex}
In $W_4^6$,
\[(\delta_0^4,\delta_2^4,\delta_4^4)=\left(
\raisebox{-0.4cm}{\begin{tikzpicture}[scale=0.45]
\draw (0,0) -- (5,0);
\draw (0,2) -- (5,2);

\foreach \i in {2,3,4,5} {
\node at (\i,0) [circle,fill,inner sep=0.9pt] {};
}

\foreach \i in {0,...,5} {
\node at (\i,2) [circle,fill,inner sep=0.9pt] {};
}

\foreach \i in {2,3,4,5} {
\draw (\i,0) -- (\i,2);
}

\foreach \i in {0} {
\draw (\i,2) .. controls (\i,1) and (\i+1,1) .. (\i+1,2);
}

\end{tikzpicture}}
\ ,
\raisebox{-0.4cm}{\begin{tikzpicture}[scale=0.45]
\draw (0,0) -- (5,0);
\draw (0,2) -- (5,2);

\foreach \i in {0,1,4,5} {
\node at (\i,0) [circle,fill,inner sep=0.9pt] {};
}

\foreach \i in {0,...,5} {
\node at (\i,2) [circle,fill,inner sep=0.9pt] {};
}

\foreach \i in {0,1,4,5} {
\draw (\i,0) -- (\i,2);
}

\foreach \i in {2} {
\draw (\i,2) .. controls (\i,1) and (\i+1,1) .. (\i+1,2);
}

\node[above] at (0,2) {\footnotesize \phantom{$.$}};
\end{tikzpicture}}
\ ,
\raisebox{-0.4cm}{\begin{tikzpicture}[scale=0.45]
\draw (0,0) -- (5,0);
\draw (0,2) -- (5,2);

\foreach \i in {0,1,2,3} {
\node at (\i,0) [circle,fill,inner sep=0.9pt] {};
}

\foreach \i in {0,...,5} {
\node at (\i,2) [circle,fill,inner sep=0.9pt] {};
}

\foreach \i in {0,1,2,3} {
\draw (\i,0) -- (\i,2);
}

\foreach \i in {4} {
\draw (\i,2) .. controls (\i,1) and (\i+1,1) .. (\i+1,2);
}

\node[above] at (0,2) {\footnotesize \phantom{$.$}};
\end{tikzpicture}}\right)
.
\]
Consider the element
\[x=
\raisebox{-0.4cm}{\begin{tikzpicture}[scale=0.45]
\draw (0,0) -- (9,0);
\draw (0,2) -- (9,2);

\foreach \i in {0,5,6,7,8,9} {
\node at (\i,0) [circle,fill,inner sep=0.9pt] {};
}

\foreach \i in {0,...,9} {
\node at (\i,2) [circle,fill,inner sep=0.9pt] {};
}

\foreach \i in {0,5,6,7,8,9} {
\draw (\i,0) -- (\i,2);
}

\foreach \i in {1,3} {
\draw (\i,2) .. controls (\i,1) and (\i+1,1) .. (\i+1,2);
}

\node[above] at (0,2) {\footnotesize \phantom{$.$}};
\end{tikzpicture}}
\in W_6^{10}.
\]
The diagram $x\delta_2^4$ entails joining the third and fourth lower leftmost points of $x$, yielding
\[x\delta_2^4=
\raisebox{-0.4cm}{\begin{tikzpicture}[scale=0.45]
\draw (0,0) -- (9,0);
\draw (0,4) -- (9,4);

\foreach \i in {0,5,8,9} {
\node at (\i,0) [circle,fill,inner sep=0.9pt] {};
}

\foreach \i in {0,5,6,7,8,9} {
\node at (\i,2) [circle,fill,inner sep=0.9pt] {};
}

\foreach \i in {0,...,9} {
\node at (\i,4) [circle,fill,inner sep=0.9pt] {};
}

\foreach \i in {0,5,6,7,8,9} {
\draw (\i,2) -- (\i,4);
}

\foreach \i in {0,5,8,9} {
\draw (\i,0) -- (\i,2);
}

\foreach \i in {1,3} {
\draw (\i,4) .. controls (\i,3) and (\i+1,3) .. (\i+1,4);
}
\draw (6,2) .. controls (6,1) and (7,1) .. (7,2);

\node[above] at (0,2) {\footnotesize \phantom{$.$}};
\end{tikzpicture}}
\ =
\raisebox{-0.4cm}{\begin{tikzpicture}[scale=0.45]
\draw (0,0) -- (9,0);
\draw (0,2) -- (9,2);

\foreach \i in {0,5,8,9} {
\node at (\i,0) [circle,fill,inner sep=0.9pt] {};
}

\foreach \i in {0,...,9} {
\node at (\i,2) [circle,fill,inner sep=0.9pt] {};
}

\foreach \i in {0,5,8,9} {
\draw (\i,0) -- (\i,2);
}

\foreach \i in {1,3,6} {
\draw (\i,2) .. controls (\i,1) and (\i+1,1) .. (\i+1,2);
}

\node[above] at (0,2) {\footnotesize \phantom{$.$}};
\end{tikzpicture}}
.
\]
In particular, we have $\phi^{10}_4(x)=x(\delta_0^4-\delta_2^4+\delta_4^4)$, so
\[\phi^{10}_4(x)=
\raisebox{-0.4cm}{\begin{tikzpicture}[scale=0.45]
\draw (0,0) -- (9,0);
\draw (0,2) -- (9,2);

\foreach \i in {6,7,8,9} {
\node at (\i,0) [circle,fill,inner sep=0.9pt] {};
}

\foreach \i in {0,...,9} {
\node at (\i,2) [circle,fill,inner sep=0.9pt] {};
}

\foreach \i in {6,7,8,9} {
\draw (\i,0) -- (\i,2);
}

\foreach \i in {1,3} {
\draw (\i,2) .. controls (\i,1) and (\i+1,1) .. (\i+1,2);
}

\draw (0,2) .. controls (0,0) and (5,0) .. (5,2);

\node[above] at (0,2) {\footnotesize \phantom{$.$}};
\end{tikzpicture}}
\ -
\raisebox{-0.4cm}{\begin{tikzpicture}[scale=0.45]
\draw (0,0) -- (9,0);
\draw (0,2) -- (9,2);

\foreach \i in {0,5,8,9} {
\node at (\i,0) [circle,fill,inner sep=0.9pt] {};
}

\foreach \i in {0,...,9} {
\node at (\i,2) [circle,fill,inner sep=0.9pt] {};
}

\foreach \i in {0,5,8,9} {
\draw (\i,0) -- (\i,2);
}

\foreach \i in {1,3,6} {
\draw (\i,2) .. controls (\i,1) and (\i+1,1) .. (\i+1,2);
}

\node[above] at (0,2) {\footnotesize \phantom{$.$}};
\end{tikzpicture}}
\ +
\raisebox{-0.4cm}{\begin{tikzpicture}[scale=0.45]
\draw (0,0) -- (9,0);
\draw (0,2) -- (9,2);

\foreach \i in {0,5,6,7} {
\node at (\i,0) [circle,fill,inner sep=0.9pt] {};
}

\foreach \i in {0,...,9} {
\node at (\i,2) [circle,fill,inner sep=0.9pt] {};
}

\foreach \i in {0,5,6,7} {
\draw (\i,0) -- (\i,2);
}

\foreach \i in {1,3,8} {
\draw (\i,2) .. controls (\i,1) and (\i+1,1) .. (\i+1,2);
}

\node[above] at (0,2) {\footnotesize \phantom{$.$}};
\end{tikzpicture}}.
\]
\end{ex}
\begin{prop}\label{welldefn}
The map $\phi_\ell^n$ is a well-defined homomorphism between standard modules.
\end{prop}
\begin{proof}
We first show that $\phi_\ell^n(x)=0$ when $x$ is a non-monic string diagram. Let $k$ be the number of caps in $x$. Observe that a cup is created if $x\delta_i^\ell$ joins two throughlines of $x$; otherwise, if $x\delta_i^\ell$ joins a cap with a throughline or another cap, then exactly one cap is removed. Hence, the diagram $x\delta_i^\ell$ has at most $k-1$ caps. As a result, if $k\ge2$, then $\phi_\ell^n(x)=0$.

Now suppose that $k=1$. Then $x$ is of the form
\[x=
\raisebox{-0.9cm}{\begin{tikzpicture}[scale=0.45]
\draw (0,0) -- (17,0);
\draw (0,2) -- (17,2);

\draw (9,0) .. controls (9,1) and (10,1) .. (10,0);

\foreach \i in {2,4,8,9,10,11,14} {
\node at (\i,0) [circle,fill,inner sep=0.9pt] {};
}

\foreach \i in {2,4,8,11,14} {
\node at (\i,2) [circle,fill,inner sep=0.9pt] {};
}

\foreach \i in {2,4,8,11,14} {
\draw (\i,0) -- (\i,2);
}

\node[below] at (2,0) {\footnotesize $1$};
\node[below] at (4,0) {\footnotesize $2$};
\node[below] at (9,0) {\footnotesize $j$};
\node[below] at (14,0) {\footnotesize $\ell+2$};
\node[below] at (1,2) {\footnotesize $w_1$};
\node[below] at (3,2) {\footnotesize $w_2$};
\node[below] at (9.5,2) {\footnotesize $w_{j}$};
\node[below] at (15.5,2) {\footnotesize $w_{\ell+1}$};

\node at (6.1,1) {\Large $\cdots$};
\node at (12.6,1) {\Large $\cdots$};
\end{tikzpicture}},
\]
where the sole cap connects the $j$th and $(j+1)$th leftmost points on the bottom and the $w_i$ are subdiagrams consisting only of nested cups. We now consider the parity of $j$.

If $j$ is odd, then in order for $x\delta_i^\ell$ to feature no caps while $i$ is even we must have $i=j-1$. Hence $x\delta_i^\ell=0$ for all even $i$ such that $i\neq j-1$. However, observe that $x\delta_{j-1}^\ell$ is formed by joining the $j$th and $(j+1)$th leftmost points on the bottom, thus completing a closed loop and vanishing due to the specialization $\beta=0$. Since $\phi^n_\ell(x)$ is a linear combination of the $x\delta_i^\ell$ restricted to even values of $i$, it follows that $\phi^n_\ell(x)=0$.

Otherwise, if $j$ is even, then for $x\delta_i^\ell$ to have no caps while $i$ is even we must have $i\in\{j-2,j\}$. Hence, we have $x\delta_i^\ell=0$ for all even $i\notin\{j-2,j\}$, so it follows that \[\phi_\ell^n(x)=(-1)^{j/2-1}x(\delta_{j-2}^\ell-\delta_j^\ell)^.\] However, note that \[x\delta^\ell_{j-2}=\raisebox{-0.4cm}{\begin{tikzpicture}[scale=0.45]
\draw (0,0) -- (17,0);
\draw (0,2) -- (17,2);

\draw (8,2) .. controls (8,0) and (9,1) .. (10,0);

\foreach \i in {2,4,10,11,14} {
\node at (\i,0) [circle,fill,inner sep=0.9pt] {};
}

\foreach \i in {2,4,8,11,14} {
\node at (\i,2) [circle,fill,inner sep=0.9pt] {};
}

\foreach \i in {2,4,11,14} {
\draw (\i,0) -- (\i,2);
}

\node[below] at (1,2) {\footnotesize $w_1$};
\node[below] at (3,2) {\footnotesize $w_2$};
\node[below] at (9.5,2) {\footnotesize $w_{j}$};
\node[below] at (15.5,2) {\footnotesize $w_{\ell+1}$};

\node at (6.1,1) {\Large $\cdots$};
\node at (12.6,1) {\Large $\cdots$};
\end{tikzpicture}}\] and \[x\delta^\ell_{j}=\raisebox{-0.4cm}{\begin{tikzpicture}[scale=0.45]
\draw (0,0) -- (17,0);
\draw (0,2) -- (17,2);

\draw (9,0) .. controls (10,1) and (11,0) .. (11,2);

\foreach \i in {2,4,8,9,14} {
\node at (\i,0) [circle,fill,inner sep=0.9pt] {};
}

\foreach \i in {2,4,8,11,14} {
\node at (\i,2) [circle,fill,inner sep=0.9pt] {};
}

\foreach \i in {2,4,8,14} {
\draw (\i,0) -- (\i,2);
}

\node[below] at (1,2) {\footnotesize $w_1$};
\node[below] at (3,2) {\footnotesize $w_2$};
\node[below] at (9.5,2) {\footnotesize $w_{j}$};
\node[below] at (15.5,2) {\footnotesize $w_{\ell+1}$};

\node at (6.1,1) {\Large $\cdots$};
\node at (12.6,1) {\Large $\cdots$};
\end{tikzpicture}}.
\] As a result \[x\delta^\ell_{j-2}=x\delta^\ell_{j}=
\raisebox{-0.9cm}{\begin{tikzpicture}[scale=0.45]
\draw (0,0) -- (16,0);
\draw (0,2) -- (16,2);

\foreach \i in {2,4,8,10,13} {
\node at (\i,0) [circle,fill,inner sep=0.9pt] {};
}

\foreach \i in {2,4,8,10,13} {
\node at (\i,2) [circle,fill,inner sep=0.9pt] {};
}

\foreach \i in {2,4,8,10,13} {
\draw (\i,0) -- (\i,2);
}

\node[below] at (2,0) {\footnotesize $1$};
\node[below] at (4,0) {\footnotesize $2$};
\node[below] at (8,0) {\footnotesize $j-1$};
\node[below] at (10,0) {\footnotesize $j$};
\node[below] at (13,0) {\footnotesize $\ell$};
\node[below] at (1,2) {\footnotesize $w_1$};
\node[below] at (3,2) {\footnotesize $w_2$};
\node[below] at (9,2) {\footnotesize $w_{j}$};
\node[below] at (14.5,2) {\footnotesize $w_{\ell+1}$};

\node at (6.1,1) {\Large $\cdots$};
\node at (11.6,1) {\Large $\cdots$};
\end{tikzpicture}}\]
and thus $\phi^n_\ell(x)=0$. Thus, in both cases, we have $\phi_\ell^n(x)=0$, so $\phi_\ell^n$ is indeed well-defined.

Now it suffices to verify that $\phi_\ell^n$ intertwines. But this is apparent as the left action of $\TL_n(0)$ on $W_\ell^n$ operates by concatenation above, while $\phi_\ell^n$ acts by concatenation below.
\end{proof}
\begin{rem}
    In \cite[Theorem~3.3]{BNS}, the authors construct explicit homomorphisms between cell modules of Webster's diagrammatic algebras. These are analogous to our construction to the $\phi^n_\ell$. However, their theorem is not directly applicable as they have (in their notation; in particular, $\ell$ means something different) an additional assumption $e>h\ell$, which does not hold for us since here $e=1$, $h=2$, and $\ell=1$. 
\end{rem}

\begin{prop}\label{comp0}
The composition $\phi_{\ell-2}^n\circ\phi_\ell^n=0$ holds.
\end{prop}
\begin{proof}
For a basis element $x\in W^n_{\ell+2}$, \[\phi_{\ell-2}^n(\phi_\ell^n(x))=\phi_{\ell-2}^n\left(\sum_{i=0}^{\ell/2}(-1)^ix\delta^\ell_{2i}\right)=\sum_{i=0}^{\ell/2}\sum_{j=0}^{\ell/2-1}(-1)^{i+j}x\delta_{2i}^\ell\delta_{2j}^{\ell-2}.\] For $i>j$ we have $x\delta_{2i}^\ell\delta_{2j}^{\ell-2}=x\delta_{2j}^\ell\delta_{2i-2}^{\ell-2}$, while for $i\le j$ we have $x\delta_{2i}^\ell\delta_{2j}^{\ell-2}=x\delta_{2j+2}^\ell\delta_{2i}^{\ell-2}$. For both cases, we join the same pairs of points on the bottom edge of the diagram. By a pairing argument on the sign factor $(-1)^{i+j}$, the above double summation vanishes.
\end{proof}

Now we are ready to begin proving Theorem~\ref{thm11}. We first need the following lemma.

\begin{lemma}\label{exactness-1}
The composition $W_{\ell+1}^n\xrightarrow{g_{\ell+2}^n}W_{\ell+2}^n\xrightarrow{\phi_\ell^n}W_\ell^n\twoheadrightarrow W_\ell^n/\operatorname{im}g_\ell^n$
is an isomorphism of vector spaces.
\end{lemma}
\begin{proof}
For a basis element
\[x=\raisebox{-0.9cm}{\begin{tikzpicture}[scale=0.45]
\draw (0,0) -- (12,0);
\draw (0,2) -- (12,2);

\foreach \i in {2,4,9} {
\node at (\i,0) [circle,fill,inner sep=0.9pt] {};
}

\foreach \i in {2,4,9} {
\node at (\i,2) [circle,fill,inner sep=0.9pt] {};
}

\foreach \i in {2,4,9} {
\draw (\i,0) -- (\i,2);
}

\node[below] at (2,0) {\footnotesize $1$};
\node[below] at (4,0) {\footnotesize $2$};
\node[below] at (9,0) {\footnotesize $\ell+1$};
\node[below] at (1,2) {\footnotesize $w_1$};
\node[below] at (3,2) {\footnotesize $w_2$};
\node[below] at (10.5,2) {\footnotesize $w_{\ell+2}$};

\node at (6.6,1) {\Large $\cdots$};
\end{tikzpicture}}\in W_{\ell+1}^{n-1},\]
where the $w_i$ are subdiagrams consisting only of cups, we have
\begin{equation}\label{eq:g(x)-description}
g^n_{\ell+2}(x)=\raisebox{-0.4cm}{\begin{tikzpicture}[scale=0.45]
\draw (0,0) -- (12,0);
\draw (0,2) -- (12,2);

\foreach \i in {2,4,9,12} {
\node at (\i,0) [circle,fill,inner sep=0.9pt] {};
}

\foreach \i in {2,4,9,12} {
\node at (\i,2) [circle,fill,inner sep=0.9pt] {};
}

\foreach \i in {2,4,9,12} {
\draw (\i,0) -- (\i,2);
}

\node[below] at (1,2) {\footnotesize $w_1$};
\node[below] at (3,2) {\footnotesize $w_2$};
\node[below] at (10.5,2) {\footnotesize $w_{\ell+2}$};

\node at (6.6,1) {\Large $\cdots$};
\end{tikzpicture}}\end{equation}
We have 
\begin{equation}\label{eq:phi-g-expression}\phi_\ell^n(g_{\ell+2}^n(x))=\sum_{i=0}^{\ell/2}(-1)^ig_{\ell+2}^n(x)\delta_{2i}^\ell.\end{equation}
By \eqref{eq:g(x)-description}, we see $g_{\ell+2}^n(x)\delta_{2i}^\ell$ will always have a rightmost throughline unless $i=\frac{\ell}2$, in which case we connect the two rightmost points on the bottom of $g_{\ell+2}^n(x)$. Thus all summands of \eqref{eq:phi-g-expression} lie in the image of $g_\ell^n$ except for $g_{\ell+2}^n(x)\delta_\ell^\ell$, hence \[\tilde f(x)=\eta(\phi_\ell^n(g_{\ell+2}^n(x)))=(-1)^{\ell/2}g_{\ell+2}^n(x)\delta_\ell^\ell=(-1)^{\ell/2}
\raisebox{-0.4cm}{\begin{tikzpicture}[scale=0.45]
\draw (0,0) -- (12,0);
\draw (0,2) -- (12,2);

\foreach \i in {2,4} {
\node at (\i,0) [circle,fill,inner sep=0.9pt] {};
}

\foreach \i in {2,4,9,12} {
\node at (\i,2) [circle,fill,inner sep=0.9pt] {};
}

\foreach \i in {2,4} {
\draw (\i,0) -- (\i,2);
}

\node[below] at (1,2) {\footnotesize $w_1$};
\node[below] at (3,2) {\footnotesize $w_2$};
\node[below] at (10.5,2) {\footnotesize $w_{\ell+2}$};

\draw (9,2) .. controls (9,0.25) and (12,0.25) .. (12,2);

\node at (6.6,1) {\Large $\cdots$};
\end{tikzpicture}}.
\] 
In other words, the map $\tilde f$ simply bends the rightmost throughline of some $x\in W_{\ell+1}^{n-1}$ into the rightmost maximal arc of $x$ while leaving everything else intact. Thus $\tilde f$ is bijective.
\end{proof}

\begin{proof}[Proof of Theorem~\ref{thm11}]
Retain the notation used in Lemma~\ref{exactness-1}. The map $f=\phi_\ell^n\circ g_{\ell+2}^n$ satisfies $\im f\subset\im\phi_\ell^n$. By Lemma~\ref{exactness-1} we see $f$ is injective. Thus \[\dim\im\phi_\ell^n\ge\dim\im f\ge\dim W_{\ell+1}^{n-1},\] implying that \begin{equation}\label{eq:phi-dimension-inequality}
\dim\im\phi_\ell^n\ge\dim W_{\ell+1}^{n-1}\text{ and }\dim\im\phi_{\ell-2}^n\ge\dim W_{\ell-1}^{n-1}.
\end{equation}
On the other hand, by Proposition~\ref{comp0} we know $\phi_{\ell-2}^n\circ\phi_\ell^n=0$ for all $\ell$, so $\im \phi_{\ell}^n\subset\ker\phi_{\ell-2}^n$. In particular, rank-nullity implies that \begin{equation}\label{eq:dimension-inequality}
\dim\im\phi_{\ell-2}^n+\dim\im\phi_\ell^n\le\dim\im\phi_{\ell-2}^n+\dim\ker\phi_{\ell-2}^n=\dim W^n_\ell.\end{equation}
By Proposition~\ref{prop42} we know $\dim W^n_\ell=\dim W^{n-1}_{\ell-1}+\dim W^{n-1}_{\ell+1}$,
so by comparing with \eqref{eq:phi-dimension-inequality} we conclude the inequality in \eqref{eq:dimension-inequality} must be an equality and \[\dim\im\phi_{\ell-2}^n=\dim W_{\ell-1}^{n-1},\ \ \dim\im\phi_\ell^n=\dim W_{\ell+1}^{n-1}.\]
Since we saw above that $\im \phi_{\ell}^n\subset\ker\phi_{\ell-2}^n$, the result follows.

The classification of irreducible modules is immediate from Proposition~\ref{prop74} and \eqref{eq:thm-exact-seq}.
\end{proof}

\section{The symmetric group algebra over characteristic two}\label{56}
Let $n$ be a positive integer.  
The ring $\field[z_1,z_2,\dots,z_n]$ carries an action of $\sym_n$ by permuting the variables $z_i$, and we realize the Specht modules as submodules of this ring.

\begin{defn}\label{spechtdefn} 
For any Young tableau $t$ of shape $\lambda$, let $F_t\in\field[z_1,z_2,\dots,z_n]$ be the product of $z_i-z_j$ where $i$ and $j$ are the respective labels of cells $b_i$ and $b_j$ in the same column of $\lambda$, with $b_i$ above $b_j$. The \textit{Specht module} $S^\lambda$ is the $\field[\sym_n]$-module spanned by $F_t$ for all Younge tableu $t$ of shape $\lambda$.
\end{defn}

\begin{ex}
Let \[t_1=\vcenter{\hbox{\young(132,4)}}\ \text{ and }\ t_2=\vcenter{\hbox{\young(43,1,2)}}.\] Then $F_{t_1}=z_1-z_4$ while $F_{t_2}=(z_4-z_1)(z_4-z_2)(z_1-z_2)$.
\end{ex}

\begin{rem}
Specht modules are typically defined using Young symmetrizers. The polynomial ideal formulation from Definition~\ref{spechtdefn} was the original construction given by Specht~\cite{specht}.
\end{rem}

When $\lambda$ is a two-row partition, the Specht module $S^\lambda$ may be realized as a submodule of the finite-dimensional vector space $\field[z_1,z_2,\dots,z_n]/( z_1^2,z_2^2,\dots,z_n^2)$ rather than the infinite-dimensional $\field[z_1,z_2,\dots,z_n]$.

\begin{defn}\label{54}
Let $T^\lambda$ be the image of $S^\lambda$ under the quotient \[\field[z_1,z_2,\dots,z_n]\longrightarrow\field[z_1,z_2,\dots,z_n]/( z_1^2,z_2^2,\dots,z_n^2).\]
\end{defn}

\begin{lemma}\label{isom-wst}
For a two-row partition $\lambda\vdash n$, the $\field[\sym_n]$-modules $S^\lambda$ and $T^\lambda$ are isomorphic.
\end{lemma}
\begin{proof}
By definition $S^\lambda$ is spanned by the polynomials $F_t$ for $t$ a Young tableau of shape $\lambda=(n-k,k)$. If $u_i$ (resp. $v_i)$ is the label of the $i$th leftmost box on the top (resp. bottom) row, then
\begin{equation}\label{Fnkk}
F_t=\prod_{i=1}^k(z_{u_i}-z_{v_i}),
\end{equation}
so $F_t$ lies in the span of all squarefree monomials. Thus $S^\lambda\cap(z_1^2,z_2^2,\dots,z_n^2)=\{0\}$ and $S^\lambda\cong T^\lambda$.
\end{proof}

\begin{defn}
For $k\le\frac{n}2$, let $G^n_{n-2k}\colon W_{n-2k}^n\to T^{(n-k,k)}$ send a basis element $x\in W_{n-2k}^n$ to the product of all terms of the form $z_i-z_j$ for all $i<j$ such that the $i$th and $j$th leftmost points on top are connected by a cup.
\end{defn}

\begin{ex}
Let
\[x=
\raisebox{-0.4cm}{\begin{tikzpicture}[scale=0.45]
\draw (0,0) -- (11,0);
\draw (0,2) -- (11,2);

\foreach \i in {0,...,11} {
\node at (\i,2) [circle,fill,inner sep=0.9pt] {};
}

\foreach \i in {1,...,12} {
\node[above] at (\i-1,2) {\footnotesize $\i$};
}

\foreach \i in {4,11} {
\node at (\i,0) [circle,fill,inner sep=0.9pt] {};
\node at (\i,2) [circle,fill,inner sep=0.9pt] {};
\draw (\i,0) -- (\i,2);
}

\draw (0,2) .. controls (0,0) and (3,0) .. (3,2);
\draw (1,2) .. controls (1,1) and (2,1) .. (2,2);
\draw (5,2) .. controls (5,0) and (10,0) .. (10,2);
\draw (6,2) .. controls (6,1) and (7,1) .. (7,2);
\draw (8,2) .. controls (8,1) and (9,1) .. (9,2);

\node[above] at (0,2) {\footnotesize \phantom{$.$}};
\end{tikzpicture}}\in W_2^{12}
.
\]
Since nodes $1$ and $4$, $2$ and $3$, $6$ and $11$, $7$ and $8$, and $9$ and $10$ are connected by cups,\[G^{12}_2(x)=(z_1-z_4)(z_2-z_3)(z_6-z_{11})(z_7-z_8)(z_9-z_{10}).\]
\end{ex}

\begin{prop}\label{bij}
The map $G_{n-2k}^n\colon W_{n-2k}^n\to T^{(n-k,k)}$ is an isomorphism of $\field$-vector spaces.
\end{prop}
\begin{proof}
Let $\lambda=(n-k,k)$. By the hook length formula and Lemma~\ref{isom-wst} \[\dim T^\lambda=\dim S^\lambda=\binom{n}{k}-\binom{n}{k-1}=\dim W_{n-2k}^n,\] so it suffices to show that $G^n_{n-2k}$ is a surjection.

The vector space $T^\lambda$ is spanned by polynomials $F_t$ for Young tableau $t$ of shape $\lambda$ as in~\eqref{Fnkk}. We can represent $F_t$ using a diagram of strings using the following procedure:
\begin{itemize}
\item Draw two parallel horizontal lines, each containing $n$ equally-spaced points.
\item For all $z_i-z_j$ dividing $F_t$, connect the $i$th and $j$th leftmost upper points with a cup.
\item For any points on the upper line that are not an endpoint of a cup, connect a vertical throughline through it.
\end{itemize}
Denote the resulting diagram by $x_t$. If $x_t$ does not contains intersections between two cups, or intersections between a cup and a throughline, then $F_t=G_{n-2k}^n(x_t)$ by construction. So we deal with the problematic cases in succession.

\smallskip

\noindent\textit{Step 1.} First, we deal with intersections between cups. Because of this, we may ignore the bottom horizontal line and all throughlines in $x_t$, wrapping everything around a circle. Thus we arrive at the diagram $w$, containing $n$ evenly spaced points around a circle labeled from $1$ to $n$ such that there exists a chord from $i$ to $j$ if and only if $z_i-z_j$ divides $F_t$.

For any $c_1$, $c_2$, $c_3$, and $c_4$ that \[(z_{c_1}-z_{c_3})(z_{c_2}-z_{c_4})=(z_{c_1}-z_{c_2})(z_{c_3}-z_{c_4})+(z_{c_1}-z_{c_4})(z_{c_2}-z_{c_3})\] in $\field[z_1,z_2,\dots,z_n]/(z_1^2,z_2^2,\dots,z_n^2)$, so any intersection of chords can be resolved by
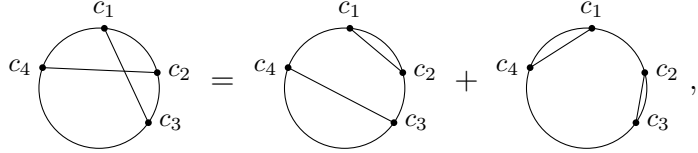
\[\raisebox{-0.8cm}{\begin{tikzpicture}[scale=0.4]
  \draw (0,0) circle (2);
  
  \foreach \i in {15,85,160,325} {
    \node at ({2*cos(\i)},{2*sin(\i)}) [circle,fill,inner sep=0.9pt] {};
  }
  \node[above] at ({2*cos(85)},{2*sin(85)}) {\footnotesize $c_1$};
  \node[right] at ({2*cos(15)},{2*sin(15)}) {\footnotesize $c_2$};
  \node[right] at ({2*cos(325)},{2*sin(325)}) {\footnotesize $c_3$};
  \node[left] at ({2*cos(160)},{2*sin(160)}) {\footnotesize $c_4$};

  \draw ({2*cos(15)},{2*sin(15)}) -- ({2*cos(160)},{2*sin(160)});
  \draw ({2*cos(85)},{2*sin(85)}) -- ({2*cos(325)},{2*sin(325)});
  
\end{tikzpicture}}=
\raisebox{-0.8cm}{\begin{tikzpicture}[scale=0.4]
  \draw (0,0) circle (2);
  
  \foreach \i in {15,85,160,325} {
    \node at ({2*cos(\i)},{2*sin(\i)}) [circle,fill,inner sep=0.9pt] {};
  }
  \node[above] at ({2*cos(85)},{2*sin(85)}) {\footnotesize $c_1$};
  \node[right] at ({2*cos(15)},{2*sin(15)}) {\footnotesize $c_2$};
  \node[right] at ({2*cos(325)},{2*sin(325)}) {\footnotesize $c_3$};
  \node[left] at ({2*cos(160)},{2*sin(160)}) {\footnotesize $c_4$};

  \draw ({2*cos(15)},{2*sin(15)}) -- ({2*cos(85)},{2*sin(85)});
  \draw ({2*cos(160)},{2*sin(160)}) -- ({2*cos(325)},{2*sin(325)});
  
\end{tikzpicture}}
+
\raisebox{-0.8cm}{\begin{tikzpicture}[scale=0.4]
  \draw (0,0) circle (2);
  
  \foreach \i in {15,85,160,325} {
    \node at ({2*cos(\i)},{2*sin(\i)}) [circle,fill,inner sep=0.9pt] {};
  }
  \node[above] at ({2*cos(85)},{2*sin(85)}) {\footnotesize $c_1$};
  \node[right] at ({2*cos(15)},{2*sin(15)}) {\footnotesize $c_2$};
  \node[right] at ({2*cos(325)},{2*sin(325)}) {\footnotesize $c_3$};
  \node[left] at ({2*cos(160)},{2*sin(160)}) {\footnotesize $c_4$};

  \draw ({2*cos(15)},{2*sin(15)}) -- ({2*cos(325)},{2*sin(325)});
  \draw ({2*cos(85)},{2*sin(85)}) -- ({2*cos(160)},{2*sin(160)});
\end{tikzpicture}}
,
\]
where the diagrams add by adding their corresponding polynomials. The number of crossings on each component above strictly decreases every time we apply the above resolution. Hence, using a finite number of resolutions, we may write $w=\sum_{i=1}^a w_i$ where the $w_i$ are all circle diagrams for which no two chords intersect. Unfurling each $w_i$ back into a string diagram, it follows that $F_t=\sum_{i=1}^a F_{t_i}$ for some $a$, where each $F_{t_i}$ is of the form given in~\eqref{Fnkk} such that their analogous string diagrams $x_{t_i}$ contain no intersections between cups.

Thus we can assume that $x_t$ has no intersections between cups.

\smallskip

\noindent\textit{Step 2.} Now we deal with intersections between cups and throughlines.
In the diagram $x_t$ if the cup corresponding to the factor $z_{c_1}-z_{c_3}$ intersects the throughline corresponding to $z_{c_2}$, then $z_{c_1}-z_{c_3}=(z_{c_1}-z_{c_2})+(z_{c_2}-z_{c_3})$, giving us the resolution
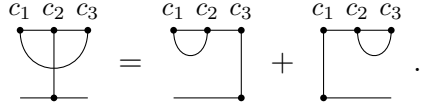
\[
\raisebox{-0.4cm}{\begin{tikzpicture}[scale=0.45]
\draw (0,0) -- (2,0);
\draw (0,2) -- (2,2);

\foreach \i in {0,...,2} {
\node at (\i,2) [circle,fill,inner sep=0.9pt] {};
}

\node at (1,0) [circle,fill,inner sep=0.9pt] {};
\draw (1,0) -- (1,2);

\draw (0,2) .. controls (0,0.5) and (2,0.5) .. (2,2);

\node[above] at (0,2) {\footnotesize \phantom{$.$}};

\node[above] at (0,2) {\footnotesize $c_1$};
\node[above] at (1,2) {\footnotesize $c_2$};
\node[above] at (2,2) {\footnotesize $c_3$};

\end{tikzpicture}}
=
\raisebox{-0.4cm}{\begin{tikzpicture}[scale=0.45]
\draw (0,0) -- (2,0);
\draw (0,2) -- (2,2);

\foreach \i in {0,...,2} {
\node at (\i,2) [circle,fill,inner sep=0.9pt] {};
}

\node at (2,0) [circle,fill,inner sep=0.9pt] {};
\draw (2,0) -- (2,2);

\draw (0,2) .. controls (0,1) and (1,1) .. (1,2);

\node[above] at (0,2) {\footnotesize \phantom{$.$}};
\node[above] at (0,2) {\footnotesize $c_1$};
\node[above] at (1,2) {\footnotesize $c_2$};
\node[above] at (2,2) {\footnotesize $c_3$};
\end{tikzpicture}}
+
\raisebox{-0.4cm}{\begin{tikzpicture}[scale=0.45]
\draw (0,0) -- (2,0);
\draw (0,2) -- (2,2);

\foreach \i in {0,...,2} {
\node at (\i,2) [circle,fill,inner sep=0.9pt] {};
}

\node at (0,0) [circle,fill,inner sep=0.9pt] {};
\draw (0,0) -- (0,2);

\draw (1,2) .. controls (1,1) and (2,1) .. (2,2);

\node[above] at (0,2) {\footnotesize \phantom{$.$}};
\node[above] at (0,2) {\footnotesize $c_1$};
\node[above] at (1,2) {\footnotesize $c_2$};
\node[above] at (2,2) {\footnotesize $c_3$};
\end{tikzpicture}}
.
\]
Again, the number of crossings on each component strictly decreases each time we use the above resolution. Thus, we may eventually write $F_t$ as the sum $\sum_{i=1}^{a'}F_{t_i}$ where each diagram $x_{t_i}$ contains no intersections between any curves. Then $F_{t_i},F_t\in\im G^n_{n-2k}$ as desired.
\end{proof}

\begin{ex}
We walk through the procedure of Proposition~\ref{bij} for $n=8$ on the polynomial $F_t=(z_1-z_4)(z_2-z_6)(z_5-z_7)\in T^{(5,3)}$. For Step 1 of Lemma~\ref{bij}, we draw the circle diagram for $F_t$ and repeatedly apply resolutions to find that
\[w=\raisebox{-1.30cm}{\begin{tikzpicture}[scale=0.4]
  \draw (0,0) circle (2);
  
  \foreach \i in {0,...,7} {
    \node at ({2*cos(45*\i)},{2*sin(45*\i)}) [circle,fill,inner sep=0.9pt] {};
  }
  \node[above] at ({2*cos(90)},{2*sin(90)}) {\footnotesize $1$};
  \node[above right] at ({2*cos(45)},{2*sin(45)}) {\footnotesize $2$};
  \node[right] at ({2*cos(0)},{2*sin(0)}) {\footnotesize $3$};
  \node[below right] at ({2*cos(-45)},{2*sin(-45)}) {\footnotesize $4$};
  \node[below] at ({2*cos(-90)},{2*sin(-90)}) {\footnotesize $5$};
  \node[below left] at ({2*cos(-135)},{2*sin(-135)}) {\footnotesize $6$};
  \node[left] at ({2*cos(180)},{2*sin(180)}) {\footnotesize $7$};
  \node[above left] at ({2*cos(135)},{2*sin(135)}) {\footnotesize $8$};

  \draw ({2*cos(90)},{2*sin(90)}) -- ({2*cos(-45)},{2*sin(-45)});
  \draw ({2*cos(45)},{2*sin(45)}) -- ({2*cos(-135)},{2*sin(-135)});
  \draw ({2*cos(-90)},{2*sin(-90)}) -- ({2*cos(180)},{2*sin(180)});
  
\end{tikzpicture}}=
\raisebox{-1.30cm}{\begin{tikzpicture}[scale=0.4]
  \draw (0,0) circle (2);
  
  \foreach \i in {0,...,7} {
    \node at ({2*cos(45*\i)},{2*sin(45*\i)}) [circle,fill,inner sep=0.9pt] {};
  }
  \node[above] at ({2*cos(90)},{2*sin(90)}) {\footnotesize $1$};
  \node[above right] at ({2*cos(45)},{2*sin(45)}) {\footnotesize $2$};
  \node[right] at ({2*cos(0)},{2*sin(0)}) {\footnotesize $3$};
  \node[below right] at ({2*cos(-45)},{2*sin(-45)}) {\footnotesize $4$};
  \node[below] at ({2*cos(-90)},{2*sin(-90)}) {\footnotesize $5$};
  \node[below left] at ({2*cos(-135)},{2*sin(-135)}) {\footnotesize $6$};
  \node[left] at ({2*cos(180)},{2*sin(180)}) {\footnotesize $7$};
  \node[above left] at ({2*cos(135)},{2*sin(135)}) {\footnotesize $8$};

  \draw ({2*cos(90)},{2*sin(90)}) -- ({2*cos(45)},{2*sin(45)});
  \draw ({2*cos(-45)},{2*sin(-45)}) -- ({2*cos(-135)},{2*sin(-135)});
  \draw ({2*cos(-90)},{2*sin(-90)}) -- ({2*cos(180)},{2*sin(180)});
  
\end{tikzpicture}}
+
\raisebox{-1.30cm}{\begin{tikzpicture}[scale=0.4]
  \draw (0,0) circle (2);
  
  \foreach \i in {0,...,7} {
    \node at ({2*cos(45*\i)},{2*sin(45*\i)}) [circle,fill,inner sep=0.9pt] {};
  }
  \node[above] at ({2*cos(90)},{2*sin(90)}) {\footnotesize $1$};
  \node[above right] at ({2*cos(45)},{2*sin(45)}) {\footnotesize $2$};
  \node[right] at ({2*cos(0)},{2*sin(0)}) {\footnotesize $3$};
  \node[below right] at ({2*cos(-45)},{2*sin(-45)}) {\footnotesize $4$};
  \node[below] at ({2*cos(-90)},{2*sin(-90)}) {\footnotesize $5$};
  \node[below left] at ({2*cos(-135)},{2*sin(-135)}) {\footnotesize $6$};
  \node[left] at ({2*cos(180)},{2*sin(180)}) {\footnotesize $7$};
  \node[above left] at ({2*cos(135)},{2*sin(135)}) {\footnotesize $8$};

  \draw ({2*cos(90)},{2*sin(90)}) -- ({2*cos(225)},{2*sin(225)});
  \draw ({2*cos(45)},{2*sin(45)}) -- ({2*cos(-45)},{2*sin(-45)});
  \draw ({2*cos(-90)},{2*sin(-90)}) -- ({2*cos(180)},{2*sin(180)});
  
\end{tikzpicture}}
\]
so
\[w=
\raisebox{-1.30cm}{\begin{tikzpicture}[scale=0.4]
  \draw (0,0) circle (2);
  
  \foreach \i in {0,...,7} {
    \node at ({2*cos(45*\i)},{2*sin(45*\i)}) [circle,fill,inner sep=0.9pt] {};
  }
  \node[above] at ({2*cos(90)},{2*sin(90)}) {\footnotesize $1$};
  \node[above right] at ({2*cos(45)},{2*sin(45)}) {\footnotesize $2$};
  \node[right] at ({2*cos(0)},{2*sin(0)}) {\footnotesize $3$};
  \node[below right] at ({2*cos(-45)},{2*sin(-45)}) {\footnotesize $4$};
  \node[below] at ({2*cos(-90)},{2*sin(-90)}) {\footnotesize $5$};
  \node[below left] at ({2*cos(-135)},{2*sin(-135)}) {\footnotesize $6$};
  \node[left] at ({2*cos(180)},{2*sin(180)}) {\footnotesize $7$};
  \node[above left] at ({2*cos(135)},{2*sin(135)}) {\footnotesize $8$};

  \draw ({2*cos(90)},{2*sin(90)}) -- ({2*cos(45)},{2*sin(45)});
  \draw ({2*cos(-45)},{2*sin(-45)}) -- ({2*cos(-90)},{2*sin(-90)});
  \draw ({2*cos(-135)},{2*sin(-135)}) -- ({2*cos(180)},{2*sin(180)});
  
\end{tikzpicture}}
+
\raisebox{-1.30cm}{\begin{tikzpicture}[scale=0.4]
  \draw (0,0) circle (2);
  
  \foreach \i in {0,...,7} {
    \node at ({2*cos(45*\i)},{2*sin(45*\i)}) [circle,fill,inner sep=0.9pt] {};
  }
  \node[above] at ({2*cos(90)},{2*sin(90)}) {\footnotesize $1$};
  \node[above right] at ({2*cos(45)},{2*sin(45)}) {\footnotesize $2$};
  \node[right] at ({2*cos(0)},{2*sin(0)}) {\footnotesize $3$};
  \node[below right] at ({2*cos(-45)},{2*sin(-45)}) {\footnotesize $4$};
  \node[below] at ({2*cos(-90)},{2*sin(-90)}) {\footnotesize $5$};
  \node[below left] at ({2*cos(-135)},{2*sin(-135)}) {\footnotesize $6$};
  \node[left] at ({2*cos(180)},{2*sin(180)}) {\footnotesize $7$};
  \node[above left] at ({2*cos(135)},{2*sin(135)}) {\footnotesize $8$};

  \draw ({2*cos(180)},{2*sin(180)}) -- ({2*cos(-45)},{2*sin(-45)});
  \draw ({2*cos(90)},{2*sin(90)}) -- ({2*cos(45)},{2*sin(45)});
  \draw ({2*cos(-135)},{2*sin(-135)}) -- ({2*cos(-90)},{2*sin(-90)});
  
\end{tikzpicture}}
+
\raisebox{-1.30cm}{\begin{tikzpicture}[scale=0.4]
  \draw (0,0) circle (2);
  
  \foreach \i in {0,...,7} {
    \node at ({2*cos(45*\i)},{2*sin(45*\i)}) [circle,fill,inner sep=0.9pt] {};
  }
  \node[above] at ({2*cos(90)},{2*sin(90)}) {\footnotesize $1$};
  \node[above right] at ({2*cos(45)},{2*sin(45)}) {\footnotesize $2$};
  \node[right] at ({2*cos(0)},{2*sin(0)}) {\footnotesize $3$};
  \node[below right] at ({2*cos(-45)},{2*sin(-45)}) {\footnotesize $4$};
  \node[below] at ({2*cos(-90)},{2*sin(-90)}) {\footnotesize $5$};
  \node[below left] at ({2*cos(-135)},{2*sin(-135)}) {\footnotesize $6$};
  \node[left] at ({2*cos(180)},{2*sin(180)}) {\footnotesize $7$};
  \node[above left] at ({2*cos(135)},{2*sin(135)}) {\footnotesize $8$};

  \draw ({2*cos(90)},{2*sin(90)}) -- ({2*cos(-90)},{2*sin(-90)});
  \draw ({2*cos(-45)},{2*sin(-45)}) -- ({2*cos(45)},{2*sin(45)});
  \draw ({2*cos(-135)},{2*sin(-135)}) -- ({2*cos(180)},{2*sin(180)});
  
\end{tikzpicture}}
+
\raisebox{-1.30cm}{\begin{tikzpicture}[scale=0.4]
  \draw (0,0) circle (2);
  
  \foreach \i in {0,...,7} {
    \node at ({2*cos(45*\i)},{2*sin(45*\i)}) [circle,fill,inner sep=0.9pt] {};
  }
  \node[above] at ({2*cos(90)},{2*sin(90)}) {\footnotesize $1$};
  \node[above right] at ({2*cos(45)},{2*sin(45)}) {\footnotesize $2$};
  \node[right] at ({2*cos(0)},{2*sin(0)}) {\footnotesize $3$};
  \node[below right] at ({2*cos(-45)},{2*sin(-45)}) {\footnotesize $4$};
  \node[below] at ({2*cos(-90)},{2*sin(-90)}) {\footnotesize $5$};
  \node[below left] at ({2*cos(-135)},{2*sin(-135)}) {\footnotesize $6$};
  \node[left] at ({2*cos(180)},{2*sin(180)}) {\footnotesize $7$};
  \node[above left] at ({2*cos(135)},{2*sin(135)}) {\footnotesize $8$};

  \draw ({2*cos(90)},{2*sin(90)}) -- ({2*cos(180)},{2*sin(180)});
  \draw ({2*cos(-45)},{2*sin(-45)}) -- ({2*cos(45)},{2*sin(45)});
  \draw ({2*cos(-90)},{2*sin(-90)}) -- ({2*cos(-135)},{2*sin(-135)});
  
\end{tikzpicture}}
.
\]
Transforming each of the four circle diagrams above into polynomials, we find that
\begin{align*}
F_t=&(z_1-z_2)(z_4-z_5)(z_6-z_7)+(z_1-z_2)(z_4-z_7)(z_6-z_6)\\
&+(z_1-z_5)(z_2-z_4)(z_6-z_7)+(z_1-z_7)(z_2-z_4)(z_5-z_6).
\end{align*}
We move on to Step 2 of Lemma~\ref{bij}. Of the four above summands, the first two correspond to valid string diagrams. The third summand requires a resolution due to an intersection of the throughline at $z_3$ with the arc due to $z_2-z_4$, and the fourth summand exhibits an intersection of the throughline at $z_8$ with the arcs due to $z_1-z_7$ and $z_2-z_4$. Applying these resolutions and putting everything together, our sum becomes
\begin{align*}
F_t&=(z_1-z_2)(z_4-z_5)(z_6-z_7)&&\text{(first summand)}\\
&\phantom{0}\quad\quad+(z_1-z_2)(z_4-z_7)(z_6-z_6)&&\text{(second summand)}\\
&\phantom{0}\quad\quad+((z_1-z_5)(z_2-z_3)(z_6-z_7)&&\text{(third summand)}\\
&\phantom{0}\quad\quad\quad\quad+(z_1-z_5)(z_3-z_4)(z_6-z_7))\\
&\phantom{0}\quad\quad+((z_1-z_4)(z_2-z_3)(z_5-z_6)&&\text{(fourth summand)}\\
&\phantom{0}\quad\quad\quad\quad+(z_2-z_3)(z_4-z_7)(z_5-z_6)\\
&\phantom{0}\quad\quad\quad\quad+(z_1-z_2)(z_3-z_4)(z_5-z_6)\\
&\phantom{0}\quad\quad\quad\quad+(z_2-z_7)(z_3-z_4)(z_5-z_6)).
\end{align*}
Now each individual summand indeed corresponds to a valid string diagram in $W_2^8$.
\end{ex}

\begin{defn}
Let $\psi_{n-2k}^n\colon T^{(n-k+1,k-1)}\to T^{(n-k,k)}$ be multiplication by $\sum_{i=1}^nz_n$.
\end{defn}

From now on, we specialize to $\field=\FF_2$ and prove the exact sequence~\eqref{eq:char2-symm-group-exact-seq}.

\begin{prop}\label{prop77}
The following diagram commutes. 
\[ \begin{tikzcd}
W_{n-2k+2}^n \arrow{r}{\phi_{n-2k}^n} \arrow[swap]{d}{G_{n-2k+2}^n} & W_{n-2k}^n \arrow{d}{G_{n-2k}^n} \\
T^{(n-k+1,k-1)} \arrow{r}{\psi_{n-2k}^n}& T^{(n-k,k)}
\end{tikzcd}
\]
\end{prop}
\begin{proof}
Let $\ell=n-2k$, and take a basis element $x\in W^n_{\ell+2}$. Number the points on the top row of the diagrammatic representation of $x$ with the integers from $1$ to $n$, going from left to right. For each $j\le\ell+2$, suppose that the $j$th leftmost throughline occurs at the point numbered with $c_j$. Then $G^n_\ell\circ\phi^n_\ell$ takes an alternating sum over connecting the $(2i-1)$th and $2i$th throughlines with a cup, which multiplies the polynomial $G_{\ell+2}^n(x)$ with the binomial $z_{c_{2i-1}}-z_{c_{2i}}$. In characteristic $2$, the alternating sum becomes \[G^n_\ell(\phi^n_\ell(x))=\sum_{i=1}^{\ell/2+1}(z_{c_{2i-1}}-z_{c_{2i}})G^n_{\ell+2}(x)=\sum_{i=1}^{\ell+2}z_{c_i}G^n_{\ell+2}(x).\] Note that $G^n_{\ell+2}(x)=\prod_{i=1}^{(n-\ell)/2-1}b_i$, where the binomials $b_i$ satisfy $\sum_{i=1}^{(n-\ell)/2-1}b_i+\sum_{i=1}^{\ell+2}z_{c_i}=\sum_{i=1}^nz_n$. Since $b_i^2=0$ in $\FF_2[z_1,z_2,\dots,z_n]/( z_1^2,z_2^2,\dots,z_n^2)$, it follows that \[\left(\sum_{i=1}^nz_n-\sum_{i=1}^{\ell+2}z_{c_i}\right)G^n_{\ell+2}(x)=\sum_{i=1}^{(n-\ell)/2-1}b_iG^n_{\ell+2}(x)=0.\] Combining the above equations implies\[\psi_\ell^n(G^n_{\ell+2}(x))=\sum_{i=1}^nz_nG^n_{\ell+2}(x)=\sum_{i=1}^{\ell+2}z_{c_i}G^n_{\ell+2}(x)=G^n_\ell(\phi^n_\ell(x)).\qedhere\]
\end{proof}

\begin{cor}
\label{exactness-3}
There is an exact sequence of $\mathbb F_2[\sym_n]$-modules
\[0\longrightarrow T^{(n)}\xrightarrow{\psi_{n-2}^n} T^{(n-1,1)}\xrightarrow{\psi_{n-4}^n}\cdots\xrightarrow{\psi_2^n}T^{(n/2+1,n/2-1)}\xrightarrow{\psi_0^n}T^{(n/2,n/2)}\longrightarrow 0.\]
\end{cor}
\begin{proof}
Follows from Theorem~\ref{thm11} and Propositions~\ref{bij} and~\ref{prop77}.
\end{proof}

\bibliographystyle{amsplain}
\bibliography{bibfile}
\end{document}